\documentclass[11pt,leqno]{article}
\usepackage{mathrsfs,amssymb,amsmath,amsthm, color}
\usepackage[dvips]{graphicx}

\makeatletter
	
	\@addtoreset{equation}{section}
\makeatother

\begin{document}

\renewcommand{\theenumi}{\rm (\roman{enumi})}
\renewcommand{\labelenumi}{\rm \theenumi}

\newtheorem{thm}{Theorem}[section]
\newtheorem{defi}[thm]{Definition}
\newtheorem{lem}[thm]{Lemma}
\newtheorem{prop}[thm]{Proposition}
\newtheorem{cor}[thm]{Corollary}
\newtheorem{exam}[thm]{Example}
\newtheorem{conj}[thm]{Conjecture}
\newtheorem{rem}[thm]{Remark}
\allowdisplaybreaks

\title{H\"older continuity and bounds for fundamental solutions to non-divergence form parabolic equations}
\author{Seiichiro Kusuoka\footnote{e-mail: kusuoka@math.tohoku.ac.jp}
\vspace{0.1in}\\
Graduate School of Science, Tohoku University\\
6-3 Aramaki Aza-Aoba, Aoba-ku Sendai 980-8578 Japan}
\maketitle

\begin{abstract}
We consider the non-degenerate second-order parabolic partial differential equations of non-divergence form with bounded measurable coefficients (not necessary continuous). Under some assumptions it is known that the fundamental solution to the equations exists uniquely, has the Gaussian bounds and is locally H\"older continuous. In the present paper, we concern the Gaussian bounds and the lower bound of the index of the H\"older continuity with respect to the initial point. We use the pinned diffusion processes for the probabilistic representation of the fundamental solutions and apply the coupling method to obtain the regularity of them. Under some assumptions weaker than the H\"older continuity of the coefficients, we obtain the Gaussian bounds and the $(1-\varepsilon)$-H\"older continuity of the fundamental solution in the initial point.
\end{abstract}

{\bf 2010 AMS Classification Numbers:} 35B65, 35K10, 60H10, 60H30, 60J60.

 \vskip0.2cm


\section{Introduction and main result}\label{section:intro}

Let $a(t,x)=(a_{ij}(t,x))$ be a symmetric $d\times d$-matrix-valued bounded measurable function on $[0,\infty ) \times {\mathbb R}^d$ which is uniformly positive definite i.e.
\begin{equation}\label{ass:a1}
\Lambda ^{-1} I \leq a(t,x) \leq \Lambda I
\end{equation}
where $\Lambda$ is a positive constant and $I$ is the unit matrix.
Let $b(t,x)=(b_i(t,x))$ be an ${\mathbb R}^d$-valued bounded measurable function on $[0,\infty ) \times {\mathbb R}^d$, and $c(t,x)$ be a bounded measurable function on $[0,\infty ) \times {\mathbb R}^d$.
Consider the following parabolic partial differential equation:
\begin{equation}\label{PDE}
\left\{ \begin{array}{rl}
\displaystyle \frac{\partial}{\partial t} u(t,x) & \displaystyle = \frac 12 \sum _{i,j=1}^d a_{ij}(t,x)\frac{\partial ^2}{\partial x_i \partial x_j} u(t,x) +  \sum _{i=1}^d b_i(t,x) \frac{\partial }{\partial x_i} u(t,x) + c(t,x) u(t,x) \!\!\!\!\!\!\!\!\!\! \\[3mm]
\displaystyle u(0,x)& =f(x)\displaystyle 
\end{array}\right.
\end{equation}
Generally, the equation (\ref{PDE}) does not have the uniqueness of the solution.
We will assume the continuity of $a$ in spatial component uniformly in $t$, and it implies the uniqueness of the week solution (see \cite{StVa}.)
In the present paper,  we always consider the cases that the uniqueness of the weak solution holds. 
Denote
\[
L_t f(x) := \frac 12 \sum _{i,j=1}^d a_{ij}(t,x)\frac{\partial ^2}{\partial x_i \partial x_j} f(x)+  \sum _{i=1}^d b_i(t,x) \frac{\partial }{\partial x_i}f(x) + c(t,x) f(x) ,\quad f\in C^2_b ({\mathbb R}^d),
\]
and the fundamental solution to (\ref{PDE}) by $p(s,x;t,y)$, i.e. $p(s,x;t,y)$ is a measurable function defined for $s,t\in [0,\infty)$ such that $s<t$, and $x,y \in {\mathbb R}^d$, which satisfies
\begin{align*}
& \frac{\partial}{\partial t} \int _{{\mathbb R}^d} f(y)p(s,\cdot  ;t,y) dy = L_t \left( \int _{{\mathbb R}^d} f(y)p(s,\cdot ;t,y) dy \right) \\
& \lim _{r\downarrow s}  \int _{{\mathbb R}^d} f(y)p(s,\cdot  ;r,y) dy = f
\end{align*}
for $s,t\in [0,\infty)$ such that $s<t$, and a continuous function $f$ with a compact support.
In the present paper, we consider the existence and the regularity of $p(0,x;t,y)$.

The problem on the regularity of the fundamental solutions to the parabolic partial differential equations with bounded measurable coefficients has the long history.
The parabolic equations of the divergence form is more investigated than that of the non-divergence form, because the variational method is applicable to them.
The H\"older continuity of the fundamental solution to ${\partial _t} u = \nabla \cdot a\nabla u$ for a matrix-valued bounded measurable function $a$ with the ellipticity condition $\Lambda ^{-1} I \leq a \leq \Lambda I$ was originally obtained by De Giorgi \cite{DG} and Nash \cite{Na} independently.
Precisely speaking, in the results the $\alpha$-H\"older continuity of the fundamental solution with some positive number $\alpha \in (0,1]$ is obtained.
The index $\alpha$ depends on many constants appeared in the Harnack inequality and so on.
These results have been extended to the case of more general equations: ${\partial _t} u = \nabla \cdot a\nabla u + b\cdot \nabla u -cu$ where $b,c$ are bounded measurable (see \cite{Ar} or \cite{St}.)
The equations with unbounded coefficients are also studied (see e.g. \cite{MePaRh}, \cite{PrEi} and \cite{PrEi2}.)
An analogy to the case of a type of nonlocal generators ( the associated stochastic processes are called stable-like processes) is given by Chen and Kumagai \cite{ChKu}.
In the results above, the index of the H\"older continuity of the fundamental solution depends on many constants appeared in the estimates, and it is difficult to calculate the exact value of the index.
Moreover, even the lower bound of the index is difficult to be obtained.

The fundamental solutions to the parabolic equations of the non-divergence form with low-regular coefficients have been studied mainly in the case of H\"older continuous coefficients.
One of the most powerful tools for the problem is L\'evi's method, and it yields the existence, the uniqueness and the H\"older continuity of the fundamental solution (see \cite{Fr}, \cite{LSU} and Chapter I of \cite{PrEi}.)
Furthermore, an a priori estimate (so-called Schauder's estimate) is known for the solutions, and the twice continuous differentiability in $x$ of the fundamental solution $p(s,x;t,y)$ to (\ref{PDE}) is obtained (see e.g. \cite{LSU}, \cite{Kr}, \cite{BoKrRo} and \cite{BoRoSh}.)
We remark that all coefficients $a,b,c$ need to be H\"older continuous to apply Schauder's estimate.
However, even in the case that $a$ is the unit matrix, when $b$ is not continuous, we cannot expect the continuous differentiability of the fundamental solution (see Remark 5.2 of Chapter 6 in \cite{KaSh}.) 

In the present paper, we consider the Gaussian estimate and the lower bound of the index for the H\"older continuity in $x$ of the fundamental solution $p(s,x;t,y)$ to (\ref{PDE}) by probabilistic approach.

Now we give the assumptions.
Let $B(x,R)$ be the open ball in ${\mathbb R}^d$ centered at $x$ with radius $R$ for $x\in {\mathbb R}^d$ and $R>0$.
We assume
\begin{equation}\label{ass:a3}
\sum _{i,j=1}^d \mathop{\rm sup}_{t\in [0,\infty )} \int _{{\mathbb R}^d} \left| \frac{\partial}{\partial x_j} a_{ij}(s,x) \right| ^\theta e^{-m|x|}dx \leq M
\end{equation}
where the derivatives are in the weak sense, $\theta$ is a constant in $[d,\infty) \cap (2,\infty)$, $m$ and $M$ are nonnegative constants.
We also assume that the continuity of $a$ in spatial component uniformly in $t$, i.e. for any $R>0$ there exists a continuous and nondecreasing function $\rho _R$ on $[0,\infty )$ such that $\rho _R (0) =0$ and
\begin{equation} \label{ass:a2}
\sup _{t\in [0,\infty )} \sup _{i,j} \left| a_{ij}(t,x) - a_{ij}(t,y)\right| \leq \rho _R (|x-y|) ,\quad x,y\in B(0;R).
\end{equation}
We remark that under the assumptions (\ref{ass:a1}) and (\ref{ass:a2}), the concerned equation (\ref{PDE}) is well-posed (see Chapter 7 in \cite{StVa}.), and for fixed $s\in [0,\infty )$ the fundamental solution $p(s,\cdot ;t,\cdot )$ exists for almost all $t\in (s,\infty)$ (see Theorem 9.1.9 in \cite{StVa}.)
However, the fundamental solution does not always exist for all $t \in (s,\infty)$ under the assumptions (\ref{ass:a1}) and (\ref{ass:a2}) without (\ref{ass:a3}) (see \cite{FaKe}.)
We remark that under the assumptions (\ref{ass:a1}), (\ref{ass:a3}) and (\ref{ass:a2}), neither existence of the fundamental solutions nor examples that the fundamental solution does not exist are known.
In the case that $a$ does not depend on time $t$, the fundamental solution exists for all $t$ (see Theorem 9.2.6 in \cite{StVa}.)
We also remark that (\ref{ass:a3}) and (\ref{ass:a2}) do not imply the local H\"older continuity of $a$ in the spatial component.

Let $p^X(s,x;t,y)$ be the fundamental solution to the parabolic equation
\[
\frac{\partial}{\partial t} u(t,x) = \frac 12 \sum _{i,j=1}^d a_{ij}(t,x)\frac{\partial ^2}{\partial x_i \partial x_j} u(t,x)
\]
and let
\begin{equation}\label{eq:defLX}
L_t^X = \frac 12 \sum _{i,j=1}^d a_{ij}(t,x)\frac{\partial ^2}{\partial x_i \partial x_j}.
\end{equation}
Let $\{ a^{(n)}(t,x)\}$ be a sequence of symmetric $d\times d$-matrix-valued bounded measurable functions in $C_b^{0,\infty}([0,\infty)\times {\mathbb R}^d)$ such that $a^{(n)}(t,x)$ converges to $a(t,x)$ for each $(t,x)\in [0,\infty ) \times {\mathbb R}^d$.
We also assume that (\ref{ass:a1}), (\ref{ass:a3}) and (\ref{ass:a2}) hold for $a^{(n)}$ instead of $a$, with the same constants $m$, $M$, $\theta$, $R$, $\rho _R$ and $\Lambda$.
Denote the fundamental solution to the parabolic equation associated with the generator
\[
\frac 12 \sum _{i,j=1}^d a_{ij}^{(n)}(t,x)\frac{\partial ^2}{\partial x_i \partial x_j}
\]
by $p^{X,{(n)}}$.
We assume the uniform Gaussian estimate for the fundamental solutions to $p^{X,(n)}$, i.e. there exist positive constants $\gamma_{\rm G}^-$, $\gamma_{\rm G}^+$, $C_{\rm G}^-$ and $C_{\rm G}^+$ such that
\begin{equation}\label{eq:gaussest}
\frac{C_{\rm G}^-}{(t-s)^{\frac d2}} \exp \left( -\frac{\gamma_{\rm G}^- |x-y|^2}{t-s}\right) \leq p^{X,(n)} (s,x;t,y) \leq \frac{C_{\rm G}^+}{(t-s)^{\frac d2}} \exp \left( -\frac{\gamma_{\rm G}^+ |x-y|^2}{t-s}\right)
\end{equation}
for $s,t \in [0,\infty)$ such that $s<t$, $x,y\in {\mathbb R}^d$, and $n\in {\mathbb N}$.
Gaussian estimates for the fundamental solutions to parabolic equations of divergence forms have been well investigated (see \cite{Ar}, \cite{Kar}, \cite{PrEi} and \cite{PrEi2}.)
However, not many results are known in the case of non-divergence forms.
A sufficient condition for the Gaussian estimate by means of Dini's continuity condition is obtained by Porper and {\`E}{\u\i}del'man (see Theorem 19 in \cite{PrEi2}).
The result includes the case of H\"older continuous coefficients.
We remark that two-sided estimates similar to Gaussian estimate for the equations with general coefficients are obtained in \cite{Es}.

Now we state the main theorem of this paper.

\begin{thm}\label{thm:main}
Assume (\ref{ass:a1}), (\ref{ass:a3}), (\ref{ass:a2}) and (\ref{eq:gaussest}).
Then, there exist constants $C_1$, $C_2$, $\gamma _1$ and $\gamma _2$ depending on $d$, $\gamma_{\rm G}^-$, $\gamma_{\rm G}^+$, $C_{\rm G}^-$, $C_{\rm G}^+$, $m$, $M$, $\theta$, $\Lambda$, $\| b\| _\infty$ and $\| c\| _\infty$ such that
\[
\frac{C_1e^{-C_1 (t-s)}}{(t-s)^{\frac d2}} \exp \left( -\frac{\gamma_1 |x-y|^2}{t-s}\right) \leq p(s,x;t,y) \leq \frac{C_2e^{C_2 (t-s)}}{(t-s)^{\frac d2}} \exp \left( -\frac{\gamma_2 |x-y|^2}{t-s}\right)
\]
for $s,t \in [0,\infty)$ such that $s<t$, and $x,y\in {\mathbb R}^d$.
Moreover, for any $R>0$ and sufficiently small $\varepsilon >0$, there exists a constant $C$ depending on $d$, $\varepsilon$, $\gamma_{\rm G}^-$, $\gamma_{\rm G}^+$, $C_{\rm G}^-$, $C_{\rm G}^+$, $m$, $M$, $\theta$, $R$, $\rho _R$, $\Lambda$, $\| b\| _\infty$ and $\| c\| _\infty$ such that
\[
|p(0,x ;t, y) - p(0,z;t,y)| \leq C t^{-d/2-1} e^{Ct} |x-z|^{1-\varepsilon}
\]
for $t\in (0,\infty )$, $x,z \in B(0;R/2)$ and $y \in {\mathbb R}^d$.
\end{thm}

The first assertion of Theorem \ref{thm:main} is the Gaussian estimate for $p$.
The advantage of the result is obtaining the Gaussian estimate of the fundamental solution to the parabolic equation of non-divergence form without the continuity of $b$ and $c$.
Such a result seems difficult to be obtained from L\'evi's method.
The second assertion of Theorem \ref{thm:main} implies that $p(0,x;t,y)$ is $(1-\varepsilon)$-H\"older continuous in $x$, and this is a clear lower bound.
The approach in this paper is mainly probabilistic.
The key method to prove Theorem \ref{thm:main} is the coupling method introduced by Lindvall and Rogers \cite{LiRo}.
This method enables us to discuss the H\"older continuity of $p(0,x;t,y)$ in $x$ from the oscillation of the diffusion processes without the regularity of the coefficients.

If $a$ is uniformly continuous in the spatial component, our proof below follows without restriction on $x,z$ and the following corollary holds.
\begin{cor}
Assume (\ref{ass:a1}), (\ref{ass:a3}), (\ref{eq:gaussest}) and there exists a continuous and nondecreasing function $\rho$ on $[0,\infty )$ such that $\rho(0) =0$ and
\[
\sup _{t\in [0,\infty )} \sup _{i,j} \left| a_{ij}(t,x) - a_{ij}(t,y)\right| \leq \rho (|x-y|) ,\quad x,y\in {\mathbb R}^d.
\]
Then, for sufficiently small $\varepsilon >0$, there exists a constant $C$ such that
\[
|p(0,x ;t, y) - p(0,z;s,y)| \leq Ct^{-d/2-1} e^{Ct} |x-z|^{1-\varepsilon}
\]
for $t\in (0,\infty )$ and $x,y,z \in {\mathbb R}^d$.
\end{cor}

The assumption (\ref{eq:gaussest}) may seem strict.
However, as mentioned above, Porper and {\`E}{\u\i}del'man obtained the Gaussian estimate for the parabolic equations with coefficients which satisfy a version of Dini's continuity condition (see Theorem 19 in \cite{PrEi2}).
From this sufficient condition and Theorem \ref{thm:main}, we have the following corollary. 

\begin{cor}\label{cor:main}
Assume (\ref{ass:a1}), (\ref{ass:a3}), and there exists a continuous and nondecreasing function $\rho$ on $[0,\infty )$ such that $\rho(0) =0$,
\begin{align}
&\int _0^1 \frac 1{r_2}\left( \int _0^{r_2} \frac{\rho(r_1)}{r_1}dr_1 \right) dr_2 < \infty , \label{eq:Dini}\\
&\sup _{t\in [0,\infty )} \sup _{i,j} \left| a_{ij}(t,x) - a_{ij}(t,y)\right| \leq \rho (|x-y|) ,\quad x,y\in {\mathbb R}^d.
\end{align}
Then, for sufficiently small $\varepsilon >0$, there exists a constant $C$ such that
\[
|p(0,x ;t, y) - p(0,z;s,y)| \leq Ct^{-d/2-1} e^{Ct} |x-z|^{1-\varepsilon}
\]
for $t\in (0,\infty )$ and $x,y,z \in {\mathbb R}^d$.
\end{cor}

We remark that; for $\alpha \in (0,1]$ and a positive constant $C$, $\rho (r)= Cr^{\alpha}$ satisfies (\ref{eq:Dini}).
Furthermore, $\rho (r)= C \min\{ 1, (-\log r)^{-\alpha} \}$ satisfies (\ref{eq:Dini}) for $\alpha \in (2,\infty )$.
We also remark that the continuity of $b$ and $c$ are not assumed in Corollary \ref{cor:main}.

The organization of the paper is as follows.

In Section \ref{sec:representation}, we prepare the probabilistic representation of the fundamental solution of (\ref{PDE}).
It should be remarked that we consider the case that $a$ is smooth in Section \ref{sec:representation}-\ref{sec:x}, and the general case is concerned only in Section \ref{sec:proof}.
The representation enable us to consider the H\"older continuity of the fundamental solution by a probabilistic way, and actuary in Sections \ref{sec:x} we prove the constant appeared in the H\"older continuity of $p(0,x;t,y)$ in $x$ depends only on the suitable constants.
The representation is obtained by the Feynman-Kac formula and the Girsanov transformation, and in the end of this section $p(s,x;t,y)$ is represented by the functional of the pinned diffusion process.

In Section \ref{sec:estimate}, we prepare some estimates.
The target of this section is Lemma \ref{lem:estE2}, which is on the integrability of a functional of the pinned diffusion process.
Generally speaking, it is much harder to see the integrability with respect to conditional probability measures than the original probability measure.
In our case, conditioning generates singularity and this fact makes the estimate difficult.
To overcome the difficulty we prepare Lemma \ref{lem:estnablap}, which is an estimate of the derivative of $p(s,x;t,y)$.
The proof of this lemma is analytic, and (\ref{ass:a3}) is assumed for the lemma.
In this section, we also have the Gaussian estimate for $p(s,x;t,y)$.

In Section \ref{sec:x} we prove that the constant appeared in the $(1-\varepsilon )$-H\"older continuity of $p(0,x;t,y)$ in $x$ depends only on the suitable constants.
This section is the main part of our argument.
To show it, we apply the coupling method to diffusion processes.
By virtue of the coupling method, the continuity problem of the fundamental solution is reduced to the problem of the local behavior of the pinned diffusion processes.
To see the local behavior, (\ref{ass:a2}) is needed.
Finally by showing an estimate of the coupling time, we obtain the $(1-\varepsilon )$-H\"older continuity of $p(0,x;t,y)$ in $x$ and the suitable dependence of the constant appeared in the H\"older continuity.

In Section \ref{sec:proof}, we consider the case of general $a$ and prove Theorem \ref{thm:main}.
The approach is only approximating $a$ by smooth ones and using the result obtained in Section \ref{sec:x}.

Throughout this paper, we denote the inner product in the Euclidean space ${\mathbb R}^d$ by $\langle \cdot ,\cdot \rangle$ and all random variables are considered on a probability space $(\Omega ,{\mathscr F}, P)$.
We denote the expectation of random variables by $E[\, \cdot \, ]$ and the expectation on the event $A\in {\mathscr F}$ (i.e. $\int _A\, \cdot\, dP$) by $E[\, \cdot \, ; A]$.
We denote the smooth functions with bounded derivatives on $S$ by $C_b^\infty (S)$ and the smooth functions on $S$ with compact supports by $C_0^\infty (S)$.

\section{Probabilistic representation of the fundamental solution}\label{sec:representation}

In this section, we assume that $a_{ij}(t,x) \in C_b^{0,\infty}([0,\infty)\times {\mathbb R}^d)$.
Define a $d\times d$-matrix-valued function $\sigma (t,x)$ by the square root of $a(t,x)$.
Then, (\ref{ass:a1}) implies that $\sigma_{ij}(t,x) \in C_b^{0,\infty}([0,\infty)\times {\mathbb R}^d)$, $a(t,x)= \sigma (t,x) \sigma (t,x)^T$ and
\begin{equation}\label{eq:sigma}
\sup _{t\in [0,\infty )}\sup _{i,j} |\sigma _{ij}(t,x)-\sigma _{ij}(t,y)| \leq C\rho _R (|x-y|) ,\quad x,y\in B(0;R),
\end{equation}
where $C$ is a constant depending on $\Lambda$.
Note that (\ref{ass:a1}) implies
\[
\Lambda ^{-1/2} I \leq \sigma (t,x) \leq \Lambda ^{1/2} I.
\]
Consider the stochastic differential equation:
\begin{equation}\label{SDEX}
\left\{ \begin{array}{rl}
\displaystyle dX_t^x & \displaystyle = \sigma (t,X_t^x)dB_t \\
\displaystyle X_0^x& \displaystyle =x .
\end{array}\right.
\end{equation}
Lipschitz continuity of $\sigma$ implies that the existence of the solution and the pathwise uniqueness hold for (\ref{SDEX}).
Let $({\mathscr F}_t)$ be the $\sigma$-field generated by $(B_s; s\in [0,t])$.
Then, the pathwise uniqueness implies that the solution $X^x_t$ is ${\mathscr F}_t$-measurable.
All stopping times appearing in this paper are associated with $({\mathscr F}_t)$.
We remark that the generator of $(X_t^x)$ is given by (\ref{eq:defLX}), and therefore the transition probability density of  of $(X_t^x)$ coincides with the fundamental solution $p^X$ of the parabolic equation generated by $(L_t^X)$.
The smoothness of $\sigma$ implies the smoothness of $p^X(s,x;t,y)$ on $(0,\infty ) \times {\mathbb R}^d \times (0,\infty ) \times {\mathbb R}^d$ (see e.g. \cite{KuSt} for the probabilistic proof, or \cite{LM} for the analytic proof).

There is a relation between the fundamental solution and the generator, as follows.
Since $p^X$ is smooth, by the definition of $p^X$ we have
\begin{equation}\label{eq:genp01}
\frac{\partial}{\partial t} p^X(s,x;t,y) = [ L_t^X p^X(s,\cdot \, ;t,y) ](x)
\end{equation}
for $s,t\in [0,\infty)$ such that $s<t$, and $x,y \in {\mathbb R}^d$.
Let $(L_t^X)^*$ be the dual operator of $L_t^X$ on $L^2({\mathbb R}^d)$.
Define $T_{s,t}^X$ and $(T_{s,t}^X)^*$ by the semigroup generated by $L_t^X$ and $(L_t^X)^*$, respectively.
Since
\[
\int _{{\mathbb R}^d} \phi (x) (T_{s,t}^X \psi )(x) dx = \int _{{\mathbb R}^d} \psi (x) \left[ (T_{s,t}^X)^* \phi \right] (x) dx ,
\]
we have
\[
\int _{{\mathbb R}^d} \phi (x) \left( \int _{{\mathbb R}^d} \psi (y)p^X(s,x;t,y) dy \right) dx = \int _{{\mathbb R}^d} \psi (x) \left( \int _{{\mathbb R}^d} \phi (y) (p^X)^* (s,x;t,y) dy \right) dx
\]
where $(p^X)^* (s,x;t,y)$ is the fundamental solution associated with $(L_t^X)^*$.
Hence, it holds that
\[
p^X(s,x;t,y) = (p^X)^* (s,y;t,x)
\]
for $s,t\in (0,\infty )$ such that $s<t$, and $x,y \in {\mathbb R}^d$.
Differentiating both sides of this equation with respect to $t$, we obtain
\begin{equation}\label{eq:genp02}
\left[ L_t^X p^X(s,\cdot \, ;t,y) \right] (x) = \left[ (L_t^X)^* (p^X)^* (s,\cdot \, ;t,x) \right] (y) = \left[ (L_t^X)^* p^X(s,x;t,\cdot \, ) \right] (y)
\end{equation}
for $s,t\in [0,\infty)$ such that $s<t$, and $x,y \in {\mathbb R}^d$.
By the Chapman-Kolmogorov equation, it holds that for $s,t,u \in [0,\infty )$ such that $u<s<t$, and $x,y \in {\mathbb R}^d$
\[
p^X(u,x;t,y) = \int _{{\mathbb R}^d} p^X(u,x;s,\xi ) p^X(s,\xi ;t,y) d\xi.
\]
Differentiating both sides of this equation with respect to $s$, we have
\[
0 = \int _{{\mathbb R}^d} \left( \frac{\partial}{\partial s} p^X(u,x;s,\xi ) \right) p^X(s,\xi ;t,y) d\xi + \int _{{\mathbb R}^d} p^X(u,x;s,\xi ) \left( \frac{\partial}{\partial s} p^X(s,\xi ;t,y) \right) d\xi
\]
for $s,u \in [0,\infty )$ such that $u<s$, and $x,y \in {\mathbb R}^d$.
Since (\ref{eq:genp01}) and (\ref{eq:genp02}) imply
\[
\frac{\partial}{\partial s} p^X(u,x;s,\xi ) = [ L_s^X p^X(u,\cdot \, ;s,\xi ) ](x) = \left[ (L_s^X)^* p^X(u,x;s,\cdot \, ) \right] (\xi ),
\]
we have for $s,t,u \in [0,\infty )$ such that $u<s<t$, and $x,y \in {\mathbb R}^d$
\begin{align*}
\int _{{\mathbb R}^d} p^X(u,x;s,\xi ) \left( \frac{\partial}{\partial s} p^X(s,\xi ;t,y) \right) d\xi
&= - \int _{{\mathbb R}^d} \left[ (L_s^X)^* p^X(u,x;s,\cdot \, ) \right] (\xi ) p^X(s,\xi ;t,y) d\xi \\
&= - \int _{{\mathbb R}^d} p^X(u,x;s,\xi ) \left[ L_s^X p^X(s,\cdot \, ;t,y)\right] (\xi) d\xi.
\end{align*}
Noting that $p^X(u,x;s,\xi )$ converges to $\delta _x (\xi)$ as $u\uparrow s$ in the sense of Schwartz distributions, we obtain
\begin{equation}\label{eq:genp}
\frac{\partial}{\partial s} p^X(s,x ;t,y) = -\left[ L_s^X p^X(s,\cdot \, ;t,y)\right] (x)
\end{equation}
for $s,t\in (0,\infty)$ such that $s<t$, and $x,y \in {\mathbb R}^d$.

Next we study the probabilistic representation of $p(s,x;t,y)$ by $p^X(s,x;t,y)$.
By the Feynman-Kac formula (see e.g. Proposition 3.10 of Chapter VIII in \cite{ReYo}) and the Girsanov transformation (see e.g. Theorem 4.2 of Chapter IV in \cite{IW}), we have the following representation of $u(t,x)$ by $X_t^x$.
\begin{equation}\label{eq:uX}\begin{array}{l}
\displaystyle u(t,x)= E\left[ f(X_t^x)\exp \left( \int _0^t \langle b_\sigma (s,X_s^x), dB_s\rangle \right. \right. \\
\displaystyle \hspace{4.5cm} \left. \left. - \frac 12 \int _0^t |b_\sigma (s,X_s^x)|^2 ds + \int _0^t c(s,X_s^x)ds\right) \right]
\end{array}\end{equation}
where $b_\sigma (t,x):=\sigma (t,x)^{-1}b(t,x)$.
For $s\leq t$ and $x\in {\mathbb R}^d$ let
\[
{\mathcal E}(s,t; X^x) := \exp \left( \int _s^t \langle b_\sigma (u,X_u^x), dB_u\rangle - \frac 12 \int _s^t |b_\sigma (u,X_u^x)|^2 du + \int _s^t c(u,X_u^x)du\right) .
\]
Then, by the definition of the fundamental solution and (\ref{eq:uX}), we obtain the probabilistic representation of the fundamental solution:
\begin{equation}\label{eq:fundamental}
p(0,x;t,y) = p^X(0,x;t,y) E^{X^x_t =y}\left[ {\mathcal E}(0,t; X^x) \right]
\end{equation}
where $P^{X_t^x = y}$ is the conditional probability measure of $P$ on $X_t^x = y$ and $E^{X^x_t=y}[\, \cdot \, ]$ is the expectation with respect to $P^{X_t^x = y}$.
Hence, to see the regularity of $p(0,x;t,y)$ in $x$, it is sufficient to see the regularity of the function $x \mapsto p^X(0,x;t,y)E^{X^x_t =y}\left[ {\mathcal E}(0,t; X^x) \right]$.
We prove Theorem \ref{thm:main} by studying the regularity of the function.
The definition of ${\mathcal E}$ implies
\begin{equation}\label{eq:lemII01}
{\mathcal E}(0,t; X^x) - {\mathcal E}(\tau \wedge t ,t; X^x) = {\mathcal E}(\tau \wedge t,t; X^x)\left(  {\mathcal E}(0,\tau \wedge t; X^x) - 1 \right)
\end{equation}
for any stopping time $\tau$ and $t\in [0,\infty)$, and by It\^o's formula we have
\begin{equation}\label{eq:ItoE}\begin{array}{l}
\displaystyle {\mathcal E}(s,t; X^x) -1 \\
\displaystyle = \int _s^t {\mathcal E}(s,u; X^x) \langle b_\sigma (u,X_u^x), dB_u\rangle + \int _s^t {\mathcal E}(s,u; X^x) c(u,X_u^x)du
\end{array}\end{equation}
 for $s,t\in [0,\infty)$ such that $s\leq t$.
We use these equations in the proof.

Now we consider the diffusion process $X^x$ pinned at $y$ at time $t$.
Let $s,t\in [0,\infty)$ such that $s<t$, $x,y \in {\mathbb R}^d$ and $\varepsilon >0$.
By the Markov property of $X$, we have for $A\in {\mathscr F}_s$
\[
P\left( A\cap\{ X_t^x \in B(y;\varepsilon )\} \right) = \int _{B(y;\varepsilon )} \left( \int _{{\mathbb R}^d} p^X(s,\xi ;t,\xi ') P\left( A\cap \{ X_s^x \in d\xi \}\right) \right) d\xi ' .
\]
Hence, we obtain
\begin{equation}\label{eq:conditional}
P^{X_t^x = y} \left( A \right) = \frac{1}{p^X(0,x;t,y)} \int _{{\mathbb R}^d} p^X(s,\xi ;t,y) P\left( A\cap\{ X_s^x \in d\xi \} \right)
\end{equation}
for $s,t\in (0,\infty )$ such that $s<t$, $A\in {\mathscr F}_s$ and $x,y \in {\mathbb R}^d$.
This formula enables us to see the generator of the pinned diffusion process.
By It\^o's formula, (\ref{eq:genp}) and (\ref{eq:conditional}) we have for $f\in C_b^2({\mathbb R}^d)$, $s,t\in [0,\infty )$ such that $s<t$, and $x,y\in {\mathbb R}^d$, 
\begin{align*}
& p^X(0,x;t,y)E^{X^x_t=y} \left[ f(X^x_s)\right] - p^X(0,x;t,y)f(x)\\
& = E\left[ f(X^x_s)p^X(s,X^x_s; t,y)\right] -E\left[ f(X^x_0)p^X(0,X^x_0; t,y)\right] \\
& = E\left[ \int _0^s (L^X_uf)(X^x_u)p^X(u,X^x_u; t,y) du\right] + E\left[ \int _0^s f(X^x_u)\left. \left( \frac{\partial}{\partial u}p^X(u,\xi; t,y)\right) \right| _{\xi= X^x_u} du\right] \\
& \quad + E\left[ \int _0^s f(X^x_u) \left( L^X_u p^X(u,\cdot\, ; t,y)\right) (X^x_u) du\right] \\
& \quad + \frac 12 E\left[ \int _0^s \left\langle \sigma (u,X^x_u)^T \nabla f(X^x_u), \sigma (u,X^x_u)^T \left( \nabla p^X(u,\cdot\, ; t,y)\right) (X^x_u) \right\rangle du\right] \\
& = p^X(0,x;t,y) \int _0^s E^{X^x_t=y} \left[ (L^X_uf)(X^x_u)\right] du \\
& \quad + \frac 12 p^X(0,x;t,y) \int _0^s E^{X^x_t=y} \left[ \left\langle \nabla f(X^x_u), a(u,X^x_u) \frac{ \left( \nabla p^X(u,\cdot\, ; t,y)\right) (X^x_u)}{p^X(u,X^x_u; t,y)} \right\rangle \right] du .
\end{align*}
Hence, the generator of $X$ pinned at $y$ at time $t$ is
\[
\frac 12 \sum _{i,j=1}^d a_{ij}(s,x)\frac{\partial ^2}{\partial x_i \partial x_j} + \left\langle \frac 12 a(s,x) \frac{\nabla p^X(s,\cdot\, ; t,y)(x)}{p^X(s,x;t,y)} , \nabla \right\rangle
\]
for $s\in [0,t)$ and $x\in {\mathbb R}^d$.
Of course, the pinned Brownian motion is an example of pinned diffusion processes (see Example 8.5 of Chapter IV in \cite{IW}).

\section{Estimates}\label{sec:estimate}

In this section we prepare some estimates for the proof of the main theorem.
Assume that $a$ is smooth and set the notation as in Section \ref{sec:representation}.

\begin{lem}\label{lem:estnablap}
Let $t\in (0,\infty )$ and $\phi$ be a nonnegative continuous function on $(0,t) \times {\mathbb R}^d$ such that $\phi (\cdot ,x) \in W_{\rm loc}^{1,1}((0,t), ds)$ for $x\in {\mathbb R}^d$ and $\phi(s,\cdot ) \in W_{\rm loc}^{1,2}({\mathbb R}^d, dx)$ for $s\in (0,t)$.
Then, for $s_1,s_2 \in (0,t)$ such that $s_1\leq s_2$
\begin{align*}
&\int _{s_1}^{s_2} \int _{{\mathbb R}^d} \frac{\left\langle a(u,\xi)\nabla \! _\xi p^X (u,\xi ;t,y), \nabla \! _\xi p^X (u,\xi ;t,y) \right\rangle}{p^X (u,\xi ;t,y) ^2 } \phi (u,\xi )d\xi du\\
&\leq C(1+ |\log (t-s_1)| ) \int _{{\mathbb R}^d} \phi (s_1,\xi) d\xi + C(t-s_1)^{-1}\int _{{\mathbb R}^d} |y-\xi |^2 \phi (s_1,\xi ) d\xi \\
&\quad + C(1+ |\log (t-s_2)| ) \int _{{\mathbb R}^d} \phi (s_2,\xi ) d\xi + C(t-s_2)^{-1}\int _{{\mathbb R}^d} |y-\xi |^2 \phi (s_2,\xi ) d\xi \\
&\quad + C \int _{s_1}^{s_2} \int _{{\mathbb R}^d} \phi (u,\xi ) d\xi du + C \int _{s_1}^{s_2} \int _{{\rm supp}\phi} \frac{\left| \nabla \! _\xi \phi (u,\xi )\right| ^2}{\phi (u,\xi )}d\xi du\\
&\quad + C\int _{s_1}^{s_2} (1+ |\log (t-u)| ) \int _{{\mathbb R}^d} \left| \frac{\partial}{\partial u} \phi (u,\xi )\right| d\xi du \\
& \quad + C\int _{s_1}^{s_2} (t-u)^{-1}\int _{{\mathbb R}^d} |y-\xi |^2 \left| \frac{\partial}{\partial u} \phi (u,\xi )\right| d\xi du \\
& \quad + C \sum _{i,j=1}^d \int _0^t \int _{{\mathbb R}^d} \left| \frac{\partial}{\partial \xi _j} a_{ij}(u,\xi ) \right| ^2 \phi (u, \xi) d\xi du
\end{align*}
where $C$ is a constant depending on  $d$, $\gamma_{\rm G}^-$, $\gamma_{\rm G}^+$, $C_{\rm G}^-$, $C_{\rm G}^+$, and $\Lambda$, and ${\rm supp}\phi$ is the support of $\phi$.
\end{lem}

\begin{rem}
If $\phi$ is a continuous function on ${\mathbb R}^d$, the Lebesgue measure of ${\rm supp} \phi \setminus \{x\in {\mathbb R}^d;\ \phi (x)>0\}$ is zero.
\end{rem}

{\it Proof of Lemma \ref{lem:estnablap}.}
It is sufficient to show the case that $\phi \in C_0^\infty ([0,t]\times {\mathbb R}^d)$, because the general case is obtained by approximation.
Let $u\in (0,t)$.
By the Leibniz rule, we have 
\begin{equation}\label{eq:a01}\begin{array}{l}
\displaystyle \frac{\partial}{\partial u} \int _{{\mathbb R}^d} \left( \log p^X (u,\xi ;t,y)\right) \phi (u,\xi ) d\xi \\[2mm]
\displaystyle = \int _{{\mathbb R}^d} \frac{\frac{\partial}{\partial u} p^X (u,\xi ;t,y)}{p^X (u,\xi ;t,y)} \phi (u,\xi ) d\xi + \int _{{\mathbb R}^d} \left( \log p^X (u,\xi ;t,y)\right) \frac{\partial}{\partial u} \phi (u,\xi ) d\xi .
\end{array}\end{equation}
The equality (\ref{eq:genp}) and the integration by parts imply
\begin{align*}
&\int _{{\mathbb R}^d} \frac{\frac{\partial}{\partial u} p^X (u,\xi ;t,y)}{p^X (u,\xi ;t,y)} \phi (u,\xi ) d\xi \\
&= -\frac 12 \sum _{i,j=1}^d \int _{{\mathbb R}^d} \frac{a_{ij}(u,\xi)\frac{\partial ^2}{\partial \xi_i \partial \xi_j} p^X (u,\xi ;t,y)}{p^X (u,\xi ;t,y)} \phi (u,\xi ) d\xi \\
&= -\frac 12 \sum _{i,j=1}^d \int _{{\mathbb R}^d} \frac{\frac{\partial}{ \partial \xi_i} \left( a_{ij}(u,\xi)\frac{\partial}{ \partial \xi_j}p^X (u,\xi ;t,y)\right) }{p^X (u,\xi ;t,y)} \phi (u,\xi ) d\xi \\
&\hspace{5cm} +  \frac 12 \sum _{i,j=1}^d \int _{{\mathbb R}^d} \frac{\left( \frac{\partial}{ \partial \xi_i} a_{ij}(u,\xi) \right) \frac{\partial}{ \partial \xi_j} p^X (u,\xi ;t,y)}{p^X (u,\xi ;t,y) } \phi (u,\xi ) d\xi \\
&= - \frac 12\int _{{\mathbb R}^d} \frac{\left\langle a(u,\xi )\nabla \! _\xi p^X (u,\xi ;t,y), \nabla \! _\xi p^X (u,\xi ;t,y) \right\rangle}{p^X (u,\xi ;t,y) ^2 } \phi (u,\xi ) d\xi \\
&\hspace{5cm} + \frac 12 \int _{{\mathbb R}^d} \left\langle \frac{a(u,\xi ) \nabla \! _\xi p^X (u,\xi ;t,y)}{p^X (u,\xi ;t,y)}, \nabla \! _\xi \phi (u,\xi ) \right\rangle d\xi \\
&\hspace{5cm} + \frac 12 \sum _{i,j=1}^d \int _{{\mathbb R}^d} \frac{\left( \frac{\partial}{ \partial \xi_i} a_{ij}(u,\xi) \right) \frac{\partial}{ \partial \xi_j} p^X (u,\xi ;t,y)}{p^X (u,\xi ;t,y) } \phi (u,\xi ) d\xi.
\end{align*}
Hence, by (\ref{eq:a01}) we have
\begin{align*}
& \frac{\partial}{\partial u} \int _{{\mathbb R}^d} \left( \log p^X (u,\xi ;t,y)\right) \phi (u,\xi ) d\xi \\
& = -\frac 12 \int _{{\mathbb R}^d} \frac{\left\langle a(u,\xi )\nabla \! _\xi p^X (u,\xi ;t,y), \nabla \! _\xi p^X (u,\xi ;t,y) \right\rangle}{p^X (u,\xi ;t,y) ^2 } \phi (u,\xi ) d\xi \\
& \quad + \frac 12 \int _{{\mathbb R}^d} \left\langle \frac{a(u,\xi ) \nabla \! _\xi p^X (u,\xi ;t,y)}{p^X (u,\xi ;t,y)}, \nabla \! _\xi \phi (u,\xi ) \right\rangle d\xi \\
& \quad + \frac 12 \sum _{i,j=1}^d \int _{{\mathbb R}^d} \frac{\left( \frac{\partial}{ \partial \xi_i} a_{ij}(u,\xi) \right) \frac{\partial}{ \partial \xi_j} p^X (u,\xi ;t,y)}{p^X (u,\xi ;t,y) } \phi (u,\xi ) d\xi \\
& \quad + \int _{{\mathbb R}^d} \left( \log p^X (u,\xi ;t,y)\right) \frac{\partial}{\partial u} \phi (u,\xi ) d\xi .
\end{align*}
Integrating the both sides from $s_1$ to $s_2$ with respect to $u$, we obtain
\begin{equation}\label{eq:a02}\begin{array}{l}
\displaystyle \frac 12 \int _{s_1}^{s_2} \int _{{\mathbb R}^d} \frac{\left\langle a(u,\xi )\nabla \! _\xi p^X (u,\xi ;t,y), \nabla \! _\xi p^X (u,\xi ;t,y) \right\rangle}{p^X (u,\xi ;t,y) ^2 } \phi (u,\xi ) d\xi du\\[4mm]
\displaystyle = \int _{{\mathbb R}^d} \left( \log p^X (s_1,\xi ;t,y)\right) \phi (s_1,\xi ) d\xi - \int _{{\mathbb R}^d} \left( \log p^X (s_2,\xi ;t,y)\right) \phi (s_2,\xi ) d\xi \\[3mm]
\displaystyle \quad + \frac 12 \int _{s_1}^{s_2} \int _{{\mathbb R}^d} \left\langle \frac{a(u,\xi ) \nabla \! _\xi p^X (u,\xi ;t,y)}{p^X (u,\xi ;t,y)}, \nabla \! _\xi \phi (u,\xi ) \right\rangle d\xi du\\[3mm]
\displaystyle \quad +  \frac 12 \sum _{i,j=1}^d  \int _{s_1}^{s_2} \int _{{\mathbb R}^d} \frac{\left( \frac{\partial}{ \partial \xi_i} a_{ij}(u,\xi) \right) \frac{\partial}{ \partial \xi_j} p^X (u,\xi ;t,y)}{p^X (u,\xi ;t,y) } \phi (u,\xi ) d\xi  du \\[5mm]
\displaystyle \quad + \int _{s_1}^{s_2} \int _{{\mathbb R}^d} \left( \log p^X (u,\xi ;t,y)\right) \frac{\partial}{\partial u} \phi (u,\xi ) d\xi du.
\end{array}\end{equation}
Now we consider the estimates of the terms in the right-hand side of this equation.
By (\ref{eq:gaussest}) we have for $s\in (0,t)$
\begin{align*}
& \left| \int _{{\mathbb R}^d} \left( \log p^X (s,\xi ;t,y)\right) \phi (s,\xi ) d\xi \right| \\
& \leq \int _{{\mathbb R}^d} \left( |\log C_{\rm G}^+| +  |\log C_{\rm G}^-| + \frac d2 |\log (t-s)| + \frac{\gamma_{\rm G}^- |y-\xi |^2}{t-s} \right) \phi (s,\xi ) d\xi .
\end{align*}
Hence, there exists a constant $C$ depending on $d$, $\gamma_{\rm G}^-$, $\gamma_{\rm G}^+$, $C_{\rm G}^-$, $C_{\rm G}^+$, and $\Lambda$ such that for $s\in (0,t)$
\begin{equation}\label{eq:a03}\begin{array}{l}
\displaystyle \left| \int _{{\mathbb R}^d} \left( \log p^X (s,\xi ;t,y)\right) \phi (s,\xi ) d\xi \right| \\
\displaystyle \leq C(1+ |\log (t-s)| ) \int _{{\mathbb R}^d} \phi (s,\xi ) d\xi + C(t-s)^{-1}\int _{{\mathbb R}^d} |y-\xi |^2 \phi (s,\xi ) d\xi .
\end{array}\end{equation}
The third term of the right-hand side of (\ref{eq:a02}) is estimated as follows:
\begin{equation}\label{eq:a04}\begin{array}{l}
\displaystyle \frac 12 \left| \int _{s_1}^{s_2} \int _{{\mathbb R}^d} \left\langle \frac{a(u,\xi ) \nabla \! _z p^X (u,\xi ;t,y)}{p^X (u,\xi ;t,y)}, \nabla \! _\xi \phi (u,\xi) \right\rangle d\xi du\right| \\[5mm]
\displaystyle \hspace{1cm} \leq \frac 18 \int _{s_1}^{s_2} \int _{{\mathbb R}^d} \frac{\left\langle a(u,\xi )\nabla \! _\xi p^X (u,\xi ;t,y), \nabla \! _\xi p^X (u,\xi ;t,y) \right\rangle}{p^X (u,\xi ;t,y) ^2 } \phi (u,\xi )d\xi du \\[3mm]
\displaystyle \hspace{1cm} \quad + 8\int _{s_1}^{s_2} \int _{{\rm supp}\phi} \frac{\left\langle a(u,\xi ) \nabla \! _\xi \phi (u,\xi ), \nabla \! _\xi \phi (u,\xi )\right\rangle}{\phi (u,\xi )}d\xi du.
\end{array}\end{equation}
To estimate the fourth term of the right-hand side of (\ref{eq:a02}), deduce
\begin{align*}
&\frac 12\left|  \sum _{i,j=1}^d  \int _{s_1}^{s_2} \int _{{\mathbb R}^d} \frac{\left( \frac{\partial}{ \partial \xi_i} a_{ij}(u,\xi) \right) \frac{\partial}{ \partial \xi_j} p^X (u,\xi ;t,y)}{p^X (u,\xi ;t,y) } \phi (u,\xi ) d\xi  du \right| \\
&\leq \frac 1{8\Lambda} \sum _{j=1}^d  \int _{s_1}^{s_2} \int _{{\mathbb R}^d} \frac{\left| \frac{\partial}{ \partial \xi_j} p^X (u,\xi ;t,y)\right| ^2}{p^X (u,\xi ;t,y)^2} \phi (u,\xi ) d\xi  du \\
&\quad + 8d \Lambda \sum _{i,j=1}^d \int _{s_1}^{s_2} \int _{{\mathbb R}^d} \left| \frac{\partial}{\partial \xi _j} a_{ij}(u,\xi )\right| ^2 \phi (u,\xi ) d\xi du.
\end{align*}
Hence, by (\ref{ass:a1}) we have
\begin{equation}\label{eq:a04-2}\begin{array}{l}
\displaystyle \frac 12\left|  \sum _{i,j=1}^d  \int _{s_1}^{s_2} \int _{{\mathbb R}^d} \frac{\left( \frac{\partial}{ \partial \xi_i} a_{ij}(t,\xi) \right) \frac{\partial}{ \partial \xi_j} p^X (u,\xi ;t,y)}{p^X (u,\xi ;t,y) } \phi (u,\xi ) d\xi  du \right| \\
\displaystyle \leq \frac 18 \int _{s_1}^{s_2} \int _{{\mathbb R}^d} \frac{\left\langle a(u,\xi )\nabla \! _\xi p^X (u,\xi ;t,y), \nabla \! _\xi p^X (u,\xi ;t,y) \right\rangle}{p^X (u,\xi ;t,y) ^2 } \phi (u,\xi ) d\xi  du \\
\displaystyle \quad + C \sum _{i,j=1}^d \int _{s_1}^{s_2} \int _{{\mathbb R}^d} \left| \frac{\partial}{\partial \xi _j} a_{ij}(u,\xi )\right| ^2 \phi (u,\xi ) d\xi du
\end{array}\end{equation}
where  $C$ depending on $d$, $\gamma_{\rm G}^-$, $\gamma_{\rm G}^+$, $C_{\rm G}^-$, $C_{\rm G}^+$ and $\Lambda$.
By using (\ref{eq:gaussest}), we calculate the final term of the right-hand side of (\ref{eq:a02}) as follows:
\begin{align*}
& \left| \int _{s_1}^{s_2} \int _{{\mathbb R}^d} \left( \log p^X (u,\xi ;t,y)\right) \frac{\partial}{\partial u} \phi (u,\xi ) d\xi du\right| \\
& \leq \int _{s_1}^{s_2} \int _{{\mathbb R}^d}  \left( |\log C_{\rm G}^+| +  |\log C_{\rm G}^-| + \frac d2 |\log (t-u)| + \frac{\gamma_{\rm G}^- |y-\xi |^2}{t-u} \right) \left| \frac{\partial}{\partial u} \phi (u,\xi )\right| d\xi du.
\end{align*}
Hence, there exists a constant $C$ depending on $d$, $\gamma_{\rm G}^-$, $\gamma_{\rm G}^+$, $C_{\rm G}^-$, $C_{\rm G}^+$ and $\Lambda$ such that
\begin{equation}\label{eq:a05}\begin{array}{l}
\displaystyle \left| \int _{{\mathbb R}^d} \left( \log p^X (s,\xi ;t,y)\right) \phi (s_1,\xi ) d\xi \right| \\
\displaystyle \leq C\int _{s_1}^{s_2} (1+ |\log (t-u)| ) \int _{{\mathbb R}^d} \left| \frac{\partial}{\partial u} \phi (u,\xi )\right| d\xi du \\
\displaystyle \quad + C\int _{s_1}^{s_2} (t-u)^{-1}\int _{{\mathbb R}^d} |y-\xi |^2 \left| \frac{\partial}{\partial u} \phi (u,\xi )\right| d\xi du.
\end{array}\end{equation}
Therefore, by (\ref{eq:a02}),  (\ref{eq:a03}),  (\ref{eq:a04}),  (\ref{eq:a04-2}) and  (\ref{eq:a05}) we obtain the assertion.
\hfill \fbox{}

\begin{lem}\label{lem:estE0}
Let $\tau _1 , \tau _2$ be stopping times such that $0 \leq \tau _1 \leq \tau _2$ almost surely.
It holds that for any $q\in {\mathbb R}$
\[
E\left[ {\mathcal E}(t\wedge \tau _1 , t\wedge \tau _2 ; X^x) ^q \right] \leq e^{C(1+q^2)t}, \ t>0,\ x,y\in {\mathbb R}^d,
\]
where $C$ is a constant depending on $d$, $\Lambda$, $\| b\| _\infty$ and $\| c\|_\infty$.
\end{lem}

\begin{proof}
Since for $t\in [0,\infty )$
\begin{align*}
&{\mathcal E}(t\wedge \tau _1 ,t\wedge \tau _2 ; X^x)\\
&= \exp \left( \int _{t\wedge \tau _1}^{t\wedge \tau _2} \langle b_\sigma (u,X_u^x), dB_u\rangle - \frac 12 \int _{t\wedge \tau _1}^{t\wedge \tau _2} |b_\sigma (u,X_u^x)|^2 du + \int _{t\wedge \tau _1}^{t\wedge \tau _2} c(u,X_u^x)du\right)\\
&= \exp \left( \int _0^{t\wedge \tau _2} \langle b_\sigma \left( s,X_u^x\right) ,dB_u\rangle - \frac {q}{16} \int _0^{t\wedge \tau _2} |b_\sigma (u,X_u^x)|^2 du \right) \\
& \quad \times \exp \left( -\int _0^{t\wedge \tau _1} \langle b_\sigma \left( s,X_u^x\right) ,dB_u\rangle - \frac {q}{16} \int _0^{t\wedge \tau _1} |b_\sigma (u,X_u^x)|^2 du \right) \\
& \quad \times \exp \left( \frac{q}{16} \int _0^{t\wedge \tau _2} |b_\sigma (u,X_u^x)|^2 du + \frac{q}{16} \int _0^{t\wedge \tau _1} |b_\sigma (u,X_u^x)|^2 du + \int _{t\wedge \tau _1}^{t\wedge \tau _2} c(u,X_u^x)du \right),
\end{align*}
by H\"older's inequality we have
\begin{equation}\label{eq:estE04}\begin{array}{l}
\displaystyle E\left[{\mathcal E}(s\wedge \tau _1 ,s\wedge \tau _2 ; X^x) ^q \right] \\[2mm]
\displaystyle \leq e^{(\|c\| _\infty + \| b_\sigma \| _\infty ^2 ) (1+q^2) t}\\
\displaystyle \quad \times E\left[ \exp \left( 2q \int _0^{t\wedge \tau _2} \langle b_\sigma \left( u,X_u^x\right) ,dB_u\rangle - \frac {q^2}8 \int _0^{t\wedge \tau _2} |b_\sigma (u,X_u^x )|^2 du \right) \right] ^{1/2}\\[3mm]
\displaystyle \quad \times E\left[ \exp \left( -2q \int _0^{t\wedge \tau _1} \langle b_\sigma \left( u,X_u^x\right) ,dB_u\rangle - \frac {q^2}8 \int _0^{t\wedge \tau _1} |b_\sigma (u,X_u^x )|^2 du \right) \right]  ^{1/2}.
\end{array}\end{equation}
On the other hand, Doob's optimal sampling theorem (see Theorem 6.1 of Chapter I in \cite{IW}) and Theorem 5.2 of Chapter III in \cite{IW} imply that
\[
\exp \left( -2q \int _0^{s\wedge \tau _1} \langle b_\sigma \left( u,X_u^x\right) ,dB_u\rangle - \frac {q^2}8 \int _0^{s\wedge \tau _1} |b_\sigma (u,X_u^x )|^2 du \right)
\]
and
\[
\exp \left( 2q \int _0^{s\wedge \tau _2} \langle b_\sigma \left( u,X_u^x\right) ,dB_u\rangle - \frac {q^2}8 \int _0^{s\wedge \tau _2} |b_\sigma (u,X_u^x )|^2 du \right)
\]
are supermartingales in $s$.
These imply that
\[
E\left[ \exp \left( -2q \int _0^{t\wedge \tau _1} \langle b_\sigma \left( u,X_u^x\right) ,dB_u\rangle - \frac {q^2}8 \int _0^{t\wedge \tau _1} |b_\sigma (u,X_u^x )|^2 du \right) \right] \leq 1,
\]
and
\[
E\left[ \exp \left( 2q \int _0^{t\wedge \tau _2} \langle b_\sigma \left( u,X_u^x\right) ,dB_u\rangle - \frac {q^2}8 \int _0^{t\wedge \tau _2} |b_\sigma (u,X_u^x )|^2 du \right) \right] \leq 1.
\]
Therefore, from (\ref{eq:estE04}) we obtain the desired estimate.
\end{proof}

\begin{lem}\label{lem:estE}
Let $\tau _1 , \tau _2$ be stopping times such that $0 \leq \tau _1 \leq \tau _2 \leq t$ almost surely.
It holds that
\begin{align*}
& p^X(0,x;t,y)E^{X_t^x=y}\left[ \int _0^t {\mathcal E}(u\wedge \tau _1 , u\wedge \tau _2 ; X^x) ^q du \right] \\
& \leq Ct^{-d/2 +1-\varepsilon} e^{C(1+q^2)t} \exp \left( -\frac{\gamma |x-y|^2}{t}\right)
\end{align*}
for $t\in (0,\infty )$, $x,y\in {\mathbb R}^d$, $q\in {\mathbb R}$ and sufficiently small $\varepsilon >0$, where $C$ and $\gamma$ are positive constants depending on  $d$, $\varepsilon$, $\gamma_{\rm G}^+$, $C_{\rm G}^+$, $\Lambda$, $\| b\| _\infty$ and $\| c\|_\infty$.
\end{lem}

\begin{proof}
In view of Fubini's theorem and (\ref{eq:conditional}), it is sufficient to show that there exist positive constants $C$ and $\gamma$ depending on $d$, $\varepsilon$, $\gamma_{\rm G}^+$, $C_{\rm G}^+$, $\Lambda$, $\| b\| _\infty$ and $\| c\|_\infty$, such that
\begin{equation}\label{eq:estE01}\begin{array}{l}
\displaystyle \int _0^t E\left[ {\mathcal E}(u\wedge \tau _1 , u\wedge \tau _2 ; X^x) ^q p^X(u,X_u;t,y) \right] du \\
\displaystyle \leq Ct^{-d/2 +1-\varepsilon} e^{C(1+q^2)t} \exp \left( -\frac{\gamma |x-y|^2}{t}\right)
\end{array}\end{equation}
for $t\in (0,\infty )$ and $x,y\in {\mathbb R}^d$.
By (\ref{eq:gaussest}) and H\"older's inequality we have
\begin{align*}
&\int _0^t E\left[ {\mathcal E}(u\wedge \tau _1 , u\wedge \tau _2 ; X^x) ^q p^X(u,X_u;t,y) \right] du\\
&\leq C_{\rm G}^+ \int _0^t E\left[ {\mathcal E}(u\wedge \tau _1 , u\wedge \tau _2 ; X^x) ^q  (t-u)^{-\frac d2} \exp \left( -\frac{\gamma_{\rm G}^+ |X_u-y|^2}{t-u}\right) \right] du\\
&\leq C_{\rm G}^+ \left( \int _0^t E\left[{\mathcal E}(u\wedge \tau _1 , u\wedge \tau _2 ; X^x) ^{(d+\varepsilon )q/\varepsilon}\right] du\right) ^{\frac{\varepsilon}{d+\varepsilon }} \\
&\quad \times \left( \int _0^t E\left[(t-u)^{-(d+\varepsilon )/2} \exp \left( -\frac{(d+\varepsilon)\gamma_{\rm G}^+ |X_u^x-y|^2}{d(t-u)}\right) \right] du \right) ^{\frac{d}{d+\varepsilon }} .
\end{align*}
Hence, in view of Lemma \ref{lem:estE0}, to show (\ref{eq:estE01}) it is sufficient to prove that 
\begin{equation}\label{eq:estE03}\begin{array}{l}
\displaystyle \int _0^t E\left[(t-u)^{-(d+\varepsilon )/2} \exp \left( -\frac{(d+\varepsilon )\gamma_{\rm G}^+ |X_u^x -y|^2}{d(t-u)}\right) \right] du \\
\displaystyle \leq Ct^{-(d+\varepsilon )/2 + 1} \exp \left( -\gamma \frac{|x-y|^2}{t}\right)
\end{array}\end{equation}
for $t\in (0,\infty)$ and $x,y\in {\mathbb R}^d$, where $C$ and $\gamma$ are constants depending on $d$, $\varepsilon$, $\gamma_{\rm G}^+$, $C_{\rm G}^+$, $\Lambda$, $\| b\| _\infty$ and $\| c\|_\infty$.

Let $\tilde \gamma := (1+\varepsilon /d)\gamma_{\rm G}^+$.
By (\ref{eq:gaussest}) again, we have for $u\in (0,t)$
\begin{align*}
&E\left[(t-u)^{-d/2} \exp \left( -\frac{ \tilde \gamma |X_u^x-y|^2}{t-u}\right) \right]\\
&\leq C_{\rm G}^+ u^{-d/2} (t-u)^{-d/2} \int _{{\mathbb R}^d} \exp \left( -\frac{ \tilde \gamma |\xi -y|^2}{t-u}\right) \exp \left( -\frac{\gamma_{\rm G}^+ |\xi -x|^2}{u}\right) d\xi \\
&= C_{\rm G}^+ u^{-d/2} (t-u)^{-d/2} \exp \left( -\frac{\gamma_{\rm G}^+ \tilde \gamma}{ u\tilde \gamma + (t-u)\gamma_{\rm G}^+}|x-y|^2\right) \\
& \quad \times \int _{{\mathbb R}^d} \exp \left( -\frac{ u\tilde \gamma + (t-u)\gamma_{\rm G}^+}{u(t-u)}\left| \xi - \frac{\gamma_{\rm G}^+ (t-u)x + \tilde \gamma u y}{ u\tilde \gamma + (t-u)\gamma_{\rm G}^+} \right| ^2 \right) d\xi \\
&= (2\pi )^{d/2}C_{\rm G}^+ \left( u\tilde \gamma + (t-u)\gamma_{\rm G}^+ \right) ^{-d/2} \exp \left( -\frac{\gamma_{\rm G}^+ \tilde \gamma}{ u\tilde \gamma + (t-u)\gamma_{\rm G}^+}|x-y|^2\right) \\
&\leq (2\pi )^{d/2}C_{\rm G}^+ {\gamma _{\rm G}^+} ^{-d/2} t^{-d/2} \exp \left( -\frac{\gamma _{\rm G} ^+ |x-y|^2}{t}\right).
\end{align*}
Hence, there exists a positive constant $C$ depending on $d$, $\gamma_{\rm G}^+$, $C_{\rm G}^+$,  $\Lambda$, $\| b\| _\infty$ and $\| c\|_\infty$, such that
\[
E\left[(t-u)^{-d/2} \exp \left( -\frac{ \tilde \gamma |X_u^x -y|^2}{t-u}\right) \right] \leq Ct^{-d/2}\exp \left( -\frac{\gamma _{\rm G}^+ |x-y|^2}{t}\right), \quad u\in (0,t).
\]
Thus, we obtain (\ref{eq:estE03}).
\end{proof}

\begin{lem}\label{lem:estE2}
Let $t\in (0,\infty)$ and $\tau _1 , \tau _2$ be stopping times such that $0 \leq \tau _1 \leq \tau _2 \leq t$ almost surely.
Then, it holds that
\begin{align*}
& p^X(0,x;t,y)E^{X_t^x=y}\left[ {\mathcal E}(s\wedge \tau _1 , s\wedge \tau _2 ; X^x) ^q \right] \\
& \leq Ct^{-d/2} e^{C(1+q^2)t} \exp \left( -\frac{\gamma |x-y|^2}{t}\right)
\end{align*}
for $t\in (0,\infty)$, $x,y\in {\mathbb R}^d$, $q\in {\mathbb R}$ and $s\in [0,t)$, where $C$ and $\gamma$ are positive constants depending on $d$, $\gamma_{\rm G}^-$, $\gamma_{\rm G}^+$, $C_{\rm G}^-$, $C_{\rm G}^+$, $m$, $M$, $\theta$, $\Lambda$, $\| b\| _\infty$ and $\| c\|_\infty$.
\end{lem}

\begin{proof}
Let  $s_1,s_2 \in (0,t)$ such that $s_1\leq s_2$.
In view of (\ref{eq:ItoE}), by (\ref{eq:conditional}) and It\^o's formula we have
\begin{align*}
& p^X(0,x;t,y) E^{X_t^x=y}\left[ {\mathcal E}( (s\wedge \tau _1)\vee s_1 , (s\wedge \tau _2) \wedge s_2 ; X^x) ^q \right] \\
&= E\left[ p^X(s_2,X_{s_2}^x;t,y) {\mathcal E}( (s\wedge \tau _1)\vee s_1 , (s\wedge \tau _2) \wedge s_2 ; X^x) ^q \right] \\
&= p^X(0,x;t,y) + E\left[ \int _{(s\wedge \tau _1)\vee s_1}^{(s\wedge \tau _2) \wedge s_2}\left. \left( \frac{\partial}{\partial u} p^X(u,\xi ;t,y)\right) \right| _{\xi =X_u^x} {\mathcal E}( (s\wedge \tau _1)\vee s_1 ,u ; X^x) ^q du\right]  \\
&\quad + E\left[ \int _{(s\wedge \tau _1)\vee s_1}^{(s\wedge \tau _2) \wedge s_2} (L^X_u p^X(u,\cdot \, ;t,y)) (X_u^x)  {\mathcal E}( (s\wedge \tau _1)\vee s_1 ,u ; X^x) ^q du \right] \\
& \quad + \frac q2 E\left[ \int _{(s\wedge \tau _1)\vee s_1}^{(s\wedge \tau _2) \wedge s_2} {\mathcal E}( (s\wedge \tau _1)\vee s_1 ,u ; X^x) ^q \right. \\
& \left. \phantom{\int _{(s\wedge \tau _1)\vee s_1}^{(s\wedge \tau _2) \wedge s_2}} \hspace{3cm} \times \left\langle \sigma (u,X_u^x)^T \left. \nabla \! _z p^X(u,z;t,y) \right| _{z=X_u^x} , b_\sigma (u,X_u^x) \right\rangle du \right] \\
& \quad + \frac {q^2}2 E\left[ \int _{(s\wedge \tau _1)\vee s_1}^{(s\wedge \tau _2) \wedge s_2} p^X(u,X_{u}^x;t,y) {\mathcal E}( (s\wedge \tau _1)\vee s_1 ,u ; X^x) ^q |b_\sigma (u,X_u^x)|^2 du \right]\\
& \quad + qE\left[ \int _{(s\wedge \tau _1)\vee s_1}^{(s\wedge \tau _2) \wedge s_2} p^X(u,X_{u}^x;t,y) {\mathcal E}( (s\wedge \tau _1)\vee s_1 ,u ; X^x) ^q c(u,X_u^x) du \right] .
\end{align*}
Hence, by (\ref{eq:genp}) we obtain
\begin{align*}
& p^X(0,x;t,y) E^{X_t^x=y}\left[ {\mathcal E}( (s\wedge \tau _1)\vee s_1 , (s\wedge \tau _2) \wedge s_2 ; X^x) ^q \right]\\
& =  p^X(0,x;t,y) \\
& \quad + \frac q2 E\left[ \int _{(s\wedge \tau _1)\vee s_1}^{(s\wedge \tau _2) \wedge s_2} {\mathcal E}( (s\wedge \tau _1)\vee s_1 ,u ; X^x) ^q \right. \\
& \left. \phantom{\int _{(s\wedge \tau _1)\vee s_1}^{(s\wedge \tau _2) \wedge s_2}} \hspace{1cm} \times \left\langle \sigma (u,X_u^x)^T \left. \nabla \! _z p^X(u,z;t,y) \right| _{z=X_u^x} , b_\sigma (u,X_u^x) \right\rangle du \right] \\
& \quad + \frac {q^2}2 E\left[ \int _{(s\wedge \tau _1)\vee s_1}^{(s\wedge \tau _2) \wedge s_2} p^X(u,X_{u}^x;t,y) {\mathcal E}( (s\wedge \tau _1)\vee s_1 ,u ; X^x) ^q |b_\sigma (u,X_u^x)|^2 du \right]\\
& \quad + qE\left[ \int _{(s\wedge \tau _1)\vee s_1}^{(s\wedge \tau _2) \wedge s_2} p^X(u,X_{u}^x;t,y) {\mathcal E}( (s\wedge \tau _1)\vee s_1 ,u ; X^x) ^q c(u,X_u^x) du \right] .
\end{align*}
In view of the boundedness of $\det \sigma $, $b$ and $c$, the desired estimate is obtained, once we show the following estimates 
\begin{align}
\label{eq:lemI1} &\begin{array}{l}\displaystyle E\left[ \int _{s_1}^{s_2} {\mathcal E}( u \wedge [(s\wedge \tau _1)\vee s_1] ,u ; X^x) ^q p^X(u,X_{u}^x;t,y) du \right] \\
\displaystyle \leq  Ct^{-d/2 +1-\varepsilon} e^{C(1+q^2)t} \exp \left( -\gamma \frac{|x-y|^2}{t}\right) \end{array}\\[3mm]
\label{eq:lemI2} &\begin{array}{l}\displaystyle E\left[ \int _{s_1}^{s_2} {\mathcal E}( u \wedge [(s\wedge \tau _1)\vee s_1] ,u ; X^x) ^q \left| \left. \nabla \! _\xi p^X(u,\xi;t,y) \right| _{\xi=X_u^x} \right| du \right] \\
\displaystyle \leq Ct^{-d/2+ (1-\varepsilon)/2} e^{C(1+q^2)t} \exp \left( -\frac{\gamma |x-y|^2}{t}\right) \end{array}
\end{align}
for sufficiently small $\varepsilon >0$, where $C$ and $\gamma$ are positive constants depending on $d$, $\varepsilon$, $\gamma_{\rm G}^-$, $\gamma_{\rm G}^+$, $C_{\rm G}^-$, $C_{\rm G}^+$, $m$, $M$, $\theta$, $\Lambda$, $\| b\| _\infty$ and $\| c\|_\infty$.
The first estimate (\ref{eq:lemI1}) follows, because by (\ref{eq:conditional}) and Lemma \ref{lem:estE} we have
\begin{align*}
&E\left[ \int _{s_1}^{s_2} {\mathcal E}( u \wedge [(s\wedge \tau _1)\vee s_1] ,u ; X^x) ^q p^X(u,X_{u}^x;t,y) du \right]\\
& = p^X(0,x;t,y) E^{X_t^x=y} \left[ \int _{s_1}^{s_2} {\mathcal E}( u \wedge [(s\wedge \tau _1)\vee s_1] ,u ; X^x) ^q du  \right] \\
& \leq Ct^{-d/2 +1-\varepsilon} e^{C(1+q^2)t} \exp \left( -\gamma \frac{|x-y|^2}{t}\right)
\end{align*}
where $C$ and $\gamma$ are positive constants depending on $d$, $\varepsilon$, $\gamma_{\rm G}^+$, $C_{\rm G}^+$, $\Lambda$, $\| b\| _\infty$ and $\| c\|_\infty$.
Now we show (\ref{eq:lemI2}).
By (\ref{eq:conditional}) and H\"older's inequality we have
\begin{align*}
& E\left[ \int _{s_1}^{s_2} {\mathcal E}( u \wedge [(s\wedge \tau _1)\vee s_1] ,u ; X^x) ^q \left| \left. \nabla \! _\xi p^X(u,\xi;t,y) \right| _{\xi=X_u^x} \right| du \right] \\
&= p^X(0,x;t,y) E^{X_t^x=y} \left[ \int _{s_1}^{s_2} {\mathcal E}( u \wedge [(s\wedge \tau _1)\vee s_1] ,u ; X^x) ^q \frac{\left| \left. \nabla \! _\xi p^X(u,\xi;t,y) \right| _{\xi=X_u^x}  \right|}{p^X(u,X_u^x;t,y)} du \right] \\
&\leq p^X(0,x;t,y) \left( \int _{s_1}^{s_2}  E^{X_t^x=y}\left[ {\mathcal E}( u \wedge [(s\wedge \tau _1)\vee s_1] ,u ; X^x) ^{2q/(1-2\varepsilon)}\right] du \right)  ^{1/2-\varepsilon } \\
&\quad \times \left( \int _{s_1}^{s_2}  [u(t-u)]^{-1/2} du \right)  ^\varepsilon \left( \int _{s_1}^{s_2}  [u(t-u)]^{\varepsilon} E^{X_t^x=y}\left[ \left( \frac{\left| \left. \nabla \! _\xi p^X(u,\xi;t,y) \right| _{\xi=X_u^x} \right|}{p^X(u,X_u^x;t,y)} \right) ^2 \right] du \right) ^{1/2} .
\end{align*}
Lemma \ref{lem:estE} and (\ref{eq:gaussest}) imply
\begin{align*}
&E\left[ \int _{s_1}^{s_2} {\mathcal E}( u \wedge [(s\wedge \tau _1)\vee s_1] ,u ; X^x) ^q \left| \left. \nabla \! _\xi p^X(u,\xi;t,y) \right| _{\xi=X_u^x} \right| du \right] \\
&\leq Ct^{-d/4 +1/2 - \varepsilon } e^{C(1+q^2)t} \exp \left( -\gamma \frac{|x-y|^2}{t}\right)\\
&\quad \times \left( p^X(0,x;t,y) \int _{s_1}^{s_2}  [u(t-u)]^{\varepsilon} E^{X_t^x=y}\left[ \left( \frac{\left| \left. \nabla \! _\xi p^X(u,\xi;t,y) \right| _{\xi=X_u^x} \right|}{p^X(u,X_u^x;t,y)} \right) ^2 \right] du \right) ^{1/2}
\end{align*}
where $C$ and $\gamma $ are positive constants depending on $d$, $\varepsilon$, $\gamma_{\rm G}^+$, $C_{\rm G}^+$, $\Lambda$, $\| b\| _\infty$ and $\| c\|_\infty$.
Hence, to show (\ref{eq:lemI2}) it is sufficient to prove
\begin{equation}\label{eq:estI00}\begin{array}{l}
\displaystyle p^X(0,x;t,y) \int _0^t [u(t-u)]^{\varepsilon /2} E^{X_t^x=y}\left[ \left( \frac{\left| \left. \nabla \! _\xi p^X(u,\xi;t,y) \right| _{\xi=X_u^x} \right|}{p^X(u,X_u^x;t,y)} \right) ^2 \right] du\\
\displaystyle \leq C\left( t^{-d/2+\varepsilon} + t^{-d/2+\varepsilon} |\log t| \right)
\end{array}\end{equation}
where $C$ is a constant depending on $d$, $\varepsilon$, $\gamma_{\rm G}^-$, $\gamma_{\rm G}^+$, $C_{\rm G}^-$, $C_{\rm G}^+$, $m$, $M$, $\theta$ and $\Lambda$.
The expression (\ref{eq:conditional}) implies
\begin{align*}
& p^X(0,x;t,y) \int _0^t [u(t-u)]^{\varepsilon /2} E^{X_t^x=y}\left[ \left( \frac{\left| \left. \nabla \! _\xi p^X(u,\xi;t,y) \right| _{\xi=X_u^x} \right|}{p^X(u,X_u^x;t,y)} \right) ^2 \right] du \\
&= \int _0^t [u(t-u)]^{\varepsilon /2} E\left[ \left( \frac{\left| \left. \nabla \! _\xi p^X(u,\xi;t,y) \right| _{\xi=X_u^x} \right|}{p^X(u,X_u^x;t,y)} \right) ^2 p^X(u,X_u^x;t,y) \right] du \\
&= \int _0^t [u(t-u)]^{\varepsilon /2} \int _{{\mathbb R}^d} \left( \frac{\left| \nabla \! _\xi p^X(u,\xi;t,y) \right|}{p^X(u,\xi;t,y)} \right) ^2 p^X(u,\xi ;t,y) p^X(0,x;u,\xi) d\xi du
\end{align*}
By (\ref{eq:gaussest}) we have
\begin{equation}\label{eq:estI01}\begin{array}{l}
\displaystyle p^X(0,x;t,y) \int _0^t [u(t-u)]^{\varepsilon /2} E^{X_t^x=y}\left[ \left( \frac{\left| \left. \nabla \! _\xi p^X(u,\xi;t,y) \right| _{\xi=X_u^x} \right|}{p^X(u,X_u^x;t,y)} \right) ^2 \right] du \\
\displaystyle \leq (C_{\rm G}^+)^2 \int _0^t  \int _{{\mathbb R}^d} \left( \frac{\left| \nabla \! _\xi p^X(u,\xi;t,y) \right|}{p^X(u,\xi;t,y)} \right) ^2\\
\displaystyle \hspace{3cm} \times [u(t-u)]^{-(d-\varepsilon)/2} \exp \left[ -\gamma _{\rm G}^+ \left( \frac{|\xi-x|^2}{u} + \frac{|y-\xi|^2}{t-u} \right) \right] d\xi du .
\end{array}\end{equation}
For fixed $t$, $x$ and $y$, let
\[
\phi(u, \xi ) :=  [u(t-u)]^{-(d-\varepsilon)/2} \exp \left[ -\gamma _{\rm G}^+ \left( \frac{|\xi-x|^2}{u} + \frac{|y-\xi|^2}{t-u} \right) \right] .
\]
Denote the surface area of the unit sphere in ${\mathbb R}^d$ by $\omega _d$ if $d\geq 2$.
In the case $d=1$, let $\omega _d =2$.
Explicit calculation implies
\begin{align*}
& \int _0^t \left( \int _{{\mathbb R}^d} e^{2m|\xi|/(\theta -2)}\phi (u,\xi ) ^{\theta /(\theta -2)} d\xi \right) ^{(\theta -2)/\theta} du \\
& = \int _0^t [u(t-u)]^{-(d-\varepsilon)/2} \\
& \quad \times \left( \int _{{\mathbb R}^d} e^{2m|\xi|/(\theta -2)} \exp \left[ -\frac{\gamma _{\rm G}^+ \theta}{\theta-2} \left( \frac{|\xi-x|^2}{u} + \frac{|y-\xi|^2}{t-u} \right) \right] d\xi \right) ^{(\theta -2)/\theta } du \\
& = \exp\left( -\frac{\gamma _{\rm G}^+|x-y|^2}{t}\right) \int _0^t [u(t-u)]^{-(d-\varepsilon)/2} \\
& \quad \times \left( \int _{{\mathbb R}^d} e^{2m|\xi|/(\theta -2)} \exp \left[ -\frac{\gamma _{\rm G}^+ \theta t}{(\theta-2)u(t-u)} \left| \xi - \frac{(t-u)x + uy}{t} \right| ^2 \right] d\xi \right) ^{(\theta -2)/\theta } du.
\end{align*}
Hence, noting that for $\mu _1 \in [0,\infty)$, $\mu _2 \in (0,\infty)$ and $\nu \in {\mathbb R}^d$
\begin{align*}
& \int _{{\mathbb R}^d} e^{\mu _1 |\xi|} \exp \left( -\mu _2 \left| \xi - \nu \right| ^2 \right) d\xi \\
& =  \int _{{\mathbb R}^d} e^{\mu _1 |\xi + \nu |} \exp \left( -\mu _2 \left| \xi \right| ^2 \right) d\xi \\
& \leq e^{\mu _1 |\nu |} \int _{{\mathbb R}^d} \exp \left( \mu _1 |\xi| -\mu _2 \left| \xi \right| ^2 \right) d\xi \\
& = \omega _d e^{\mu _1 |\nu |} \int _{(0,\infty )} r^{d-1}\exp \left( \mu _1 r -\mu _2 r ^2 \right) dr \\
& = \omega _d \mu _2 ^{-d/2} e^{\mu _1 |\nu |} \int _{(0,\infty )} r^{d-1}\exp \left( \frac{\mu _1}{\sqrt{\mu _2}} r - r ^2 \right) dr \\
& = \omega _d \mu _2 ^{-d/2}\int _0^{1+\mu _1/\sqrt{\mu _2}} r^{d-1}\exp \left( \frac{\mu _1}{\sqrt{\mu _2}} r - r ^2 \right) dr \\
& \hspace{4cm} + \omega _d \mu _2 ^{-d/2} \int _{1+\mu _1/\sqrt{\mu _2}} ^\infty r^{d-1}\exp \left( \frac{\mu _1}{\sqrt{\mu _2}} r - r ^2 \right) dr\\
& \leq  \frac{\omega _d \mu _2 ^{-d/2}}{d} \left( 1+ \frac{\mu _1}{\sqrt{\mu _2}} \right) ^d \exp \left[ \frac{\mu _1}{\sqrt{\mu _2}} \left( 1+ \frac{\mu _1}{\sqrt{\mu _2}} \right) \right] + \omega _d \mu _2 ^{-d/2} \int _{1+\mu _1/\sqrt{\mu _2}} ^\infty r^{d-1}e^{-r} dr \\
& \leq C \mu _2 ^{-d/2} \exp \left[ C \frac{\mu _1}{\sqrt{\mu _2}} \left( 1+ \frac{\mu _1}{\sqrt{\mu _2}} \right) \right]
\end{align*}
where $C$ is a constant depending on $d$, we have
\begin{align*}
& \int _0^t \left( \int _{{\mathbb R}^d} e^{2m|\xi|/(\theta -2)}\phi (u,\xi ) ^{\theta /(\theta -2)} d\xi \right) ^{(\theta -2)/\theta} du \\
& \leq C_1 t^{-d/2+d/\theta} \exp\left( -\frac{\gamma _{\rm G}^+|x-y|^2}{t}\right) \int _0^t [u(t-u)]^{-d/\theta +\varepsilon /2} \\
& \hspace{5cm} \times \exp \left[ C_1 \sqrt{u\left( 1-\frac ut \right)}\left( 1 + \sqrt{u\left( 1-\frac ut \right)}\right) \right] du \\
& \leq C_2 t^{-d/2+d/\theta} e^{C_2t}\int _0^t [u(t-u)]^{-d/\theta +\varepsilon /2}  du \\
& \leq C_3 t^{-d/2+1-d/\theta + \varepsilon /2} e^{C_2t}
\end{align*}
where $C_1 , C_2, C_3$ are constants depending on  $d$, $\varepsilon$, $\gamma_{\rm G}^+$, $m$ and $\theta$.
Hence, by H\"older's inequality and (\ref{ass:a3}) we have
\begin{align*}
& \sum _{i,j=1}^d \int _{s_1}^{s_2} \int _{{\mathbb R}^d} \left| \frac{\partial}{\partial \xi _j} a_{ij}(u,\xi )\right| ^2 \phi (u,\xi) d\xi du \\
& \leq \int _0^t \left( \int _{{\mathbb R}^d} e^{2m|\xi|/(\theta -2)}\phi (u,\xi ) ^{\theta /(\theta -2)} d\xi \right) ^{(\theta -2)/\theta } du \\
&\hspace{4cm} \times \sum _{i,j=1}^d \left( \mathop{\rm sup}_{u\in [0,t]} \int _{{\mathbb R}^d} \left| \frac{\partial}{\partial \xi _j} a_{ij}(u,\xi )\right| ^\theta e^{-m|\xi|} d\xi \right) ^{2/\theta} \\
& \leq C t^{-d/2+1-d/\theta + \varepsilon /2} e^{Ct}
\end{align*}
where $C$ is a constant depending on $d$, $\varepsilon$, $\gamma_{\rm G}^+$, $m$, $M$ and $\theta$.
On the other hand, by explicit calculation we have
\begin{align*}
&\lim _{s\downarrow 0} (1+ |\log (t-s)| ) \int _{{\mathbb R}^d} \phi (s,\xi) d\xi =0,\\
&\lim _{s\downarrow 0}  (t-s)^{-1}\int _{{\mathbb R}^d} |y-\xi|^2 \phi (s,\xi) d\xi=0,\\
&\lim _{s\uparrow t}(1+ |\log (t-s)| ) \int _{{\mathbb R}^d} \phi (s,\xi) d\xi =0,\\
&\lim _{s\uparrow t} (t-s)^{-1}\int _{{\mathbb R}^d} |y-\xi|^2 \phi (s,\xi) d\xi =0,\\
&\int _0^t \int _{{\mathbb R}^d} \frac{\left| \nabla \! _\xi \phi (u,\xi)\right| ^2}{\phi (u,\xi)} d\xi du \leq Ct^{-d/2 +\varepsilon },\\
&\int _0^t (1+ |\log (t-u)| ) \int _{{\mathbb R}^d} \left| \frac{\partial}{\partial u} \phi (u,\xi )\right| d\xi du \leq Ct^{-d/2 +\varepsilon} |\log t|\\
&\int _0^t (t-u)^{-1}\int _{{\mathbb R}^d} |y-\xi |^2 \left| \frac{\partial}{\partial u} \phi (u,\xi )\right| d\xi du \leq Ct^{-d/2+\varepsilon},
\end{align*}
where $C$ is a constant depending on $d$, $\varepsilon$ and $\gamma_{\rm G}^+$.
In view of these results, applying Lemma \ref{lem:estnablap} to (\ref{eq:estI01}), we obtain (\ref{eq:estI00}).
\end{proof}

From Lemma \ref{lem:estE2} we can easily show the Gaussian estimate for $p$ with the constants depending on the suitable constants.

\begin{prop}\label{lem:cpt}
It holds that
\[
\frac{C_1e^{-C_1 (t-s)}}{(t-s)^{\frac d2}} \exp \left( -\frac{\gamma_1 |x-y|^2}{t-s}\right) \leq p(s,x;t,y) \leq \frac{C_2e^{C_2 (t-s)}}{(t-s)^{\frac d2}} \exp \left( -\frac{\gamma_2 |x-y|^2}{t-s}\right)
\]
for $s,t \in [0,\infty)$ such that $s<t$, and $x,y\in {\mathbb R}^d$, where $\gamma_1$, $\gamma_2$, $C_1$ and $C_2$ are positive constants depending on $d$, $\gamma_{\rm G}^-$, $\gamma_{\rm G}^+$, $C_{\rm G}^-$, $C_{\rm G}^+$, $m$, $M$, $\theta$, $\Lambda$, $\| b\| _\infty$ and $\| c\| _\infty$.
\end{prop}

\begin{proof}
Since all the argument follows even if $a$, $b$ and $c$ are replaced by $a(\cdot -s, \cdot )$, $b(\cdot -s, \cdot )$ and $c(\cdot -s, \cdot )$ respectively, it is sufficient to show that there exist positive constants $\gamma_1$, $\gamma_2$, $C_1$ and $C_2$ depending on $d$, $m$, $M$, $\theta$, $\Lambda$, $\| b\| _\infty$ and $\| c\| _\infty$
\begin{equation}\label{eq:lemcpt}
C_1 t^{-d/2}e^{-C_1 t} \exp \left( -\frac{\gamma_1 |x-y|^2}{t}\right) \leq p(0,x;t,y) \leq C_2 t^{-d/2}e^{C_2 t} \exp \left( -\frac{\gamma_2 |x-y|^2}{t}\right)
\end{equation}
for $t \in (0,\infty)$ and $x,y\in {\mathbb R}^d$.
The upper estimate in (\ref{eq:lemcpt}) follows immediately from (\ref{eq:fundamental}) and Lemma \ref{lem:estE2}.

Now we prove the lower estimate in (\ref{eq:lemcpt}).
From H\"older's inequality, it follows that
\[
1 \leq E^{X_t^x=y}\left[ {\mathcal E}(s\wedge \tau _1 , s\wedge \tau _2 ; X^x) ^{-1} \right] E^{X_t^x=y}\left[ {\mathcal E}(s\wedge \tau _1 , s\wedge \tau _2 ; X^x) \right].
\]
Hence, by Lemma \ref{lem:estE2} we have
\begin{align*}
& p^X(0,x;t,y)E^{X_t^x=y}\left[ {\mathcal E}(s\wedge \tau _1 , s\wedge \tau _2 ; X^x) \right] \\
& \geq \frac{p^X(0,x;t,y) ^2 }{p^X(0,x;t,y) E^{X_t^x=y}\left[ {\mathcal E}(s\wedge \tau _1 , s\wedge \tau _2 ; X^x) ^{-1} \right]}\\
& \geq C t^{d/2} e^{-C' t}p^X(0,x;t,y) ^2,
\end{align*}
where $C$ and $C'$ are positive constants depending on $d$, $m$, $M$, $\theta$, $\Lambda$, $\| b\| _\infty$ and $\| c\| _\infty$.
This inequality, (\ref{eq:fundamental}) and (\ref{eq:gaussest}) imply the lower bound in (\ref{eq:lemcpt}).
\end{proof}

\section{The regularity of $p(0,x;t,y)$ in $x$}\label{sec:x}

Assume that $a$ is smooth and set the notation as in Section \ref{sec:representation}.
In this section, we prove the H\"older continuity of $p(0,x;t,y)$ in $x$, and the constant depends only on suitable ones.
The precise statement is as follows.

\begin{prop}\label{prop1}
For any $R>0$ and sufficiently small $\varepsilon >0$, there exists a constant $C$ depending on $d$, $\varepsilon$, $\gamma_{\rm G}^-$, $\gamma_{\rm G}^+$, $C_{\rm G}^-$, $C_{\rm G}^+$, $m$, $M$, $\theta$, $R$, $\rho _R$, $\Lambda$, $\| b\| _\infty$ and $\| c\| _\infty$ such that
\[
|p(0,x;t,y) - p(0,z;t,y)| \leq C t^{-d/2-1} e^{Ct} |x-z|^{1-\varepsilon}
\]
for $t\in (0,\infty )$, $x,z \in B(0;R/2)$ and $y \in {\mathbb R}^d$.
\end{prop}

We use the coupling method (see e.g. \cite{LiRo}, \cite{Cr}).
Let $x,z \in {\mathbb R}^d$.
According to $(X^x, B)$ defined by (\ref{SDEX}), we consider the stochastic process $Z_t^z$ defined by
\begin{equation}\label{SDEz}
\left\{ \begin{array}{rl}
\displaystyle Z_t^z & \displaystyle = z + \int _0 ^{t\wedge \tau} \sigma (s,Z_s^z)d\tilde B_s + \int _{t\wedge \tau} ^t \sigma (s,Z_s^z)dB_s\\[3mm]
\displaystyle \tilde B_t & \displaystyle = \int _0^{t\wedge \tau} \left( I -\frac{2(\sigma (s,Z_s^z)^{-1} (X_s^x-Z_s^z))\otimes (\sigma (s,Z_s^z)^{-1} (X_s^x-Z_s^z))}{|\sigma (s,Z_s^z)^{-1} (X_s^x-Z_s^z)|^2}\right) dB_s
\end{array}\right.
\end{equation}
where $\tau$ is the stopping time defined by $\tau := \inf \{ t\geq 0; X_t^x=Z_t^z\}$.
By the Lipschitz continuity of $\sigma$, $(Z^z(t), \tilde B(t); t\in [0,\tau ))$ are determined almost surely and uniquely (see Theorem 6 of Chapter V in \cite{Pr}).
Let
\[
H_t :=I -\frac{2\left[ \sigma (t,Z_t^z)^{-1} (X_t^x-Z_t^z)\right] \otimes \left[ \sigma (t,Z_t^z)^{-1} (X_t^x-Z_t^z)\right] }{|\sigma (t,Z_t^z)^{-1} (X_t^x-Z_t^z)|^2}
\]
for $t\in [0,\tau )$.
Then, $H_t$ is an orthogonal matrix for all $t\in [0,\tau )$, and hence $\tilde B_t$ is a $d$-dimensional Brownian motion for $t\in [0,\tau )$.
Hence, $(Z^z(t), \tilde B(t); t\in [0,\tau ))$ are extended to $(Z^z(t), \tilde B(t); t\in [0,\tau ])$ almost surely and uniquely.
By the Lipschitz continuity of $\sigma$ again, (\ref{SDEz}) is solved almost surely and uniquely for $t\in [\tau , \infty )$.
Thus, we obtain $(Z^z(t); t\in [0,\infty ))$ almost surely and uniquely.
From this fact we have that $Z^z_t$ is ${\mathscr F}_t$-measurable for $t \in [0,\infty )$.
Hence, if $x=z$, $X^x$ and $Z^z$ has the same law.
Moreover, $X^x_t = Z^z_t$ for $t\in [\tau ,\infty )$ almost surely.

\begin{lem}\label{lem:estprob1}
For $R>0$ and sufficiently small $\varepsilon >0$, there exist positive constants $C$ and $c_0$ depending on $d$, $\varepsilon$, $R$, $\rho _R$ and $\Lambda$ such that
\begin{equation}\label{eq:estprob1-0}
E[t\wedge \tau] \leq C(1+t^2)|x-z|^{1-\varepsilon}
\end{equation}
for $t\in [0, \infty )$ and $x,z \in B(0;R/2)$ such that $|x-z| \leq c_0$.
\end{lem}

\begin{proof}
Let $R>0$ and $x,z \in B(0;R/2)$.
Define
\[
\xi _t := X_t^x-Z_t^z, \quad \alpha _t := \sigma (t,X_t^x) - \sigma (t,Z_t^z)H_t.
\]
Then, by It\^o's formula we have for $t\in [0,\tau)$
\begin{equation}\label{eq:xi}
d(|\xi _t|) = \left\langle \frac{\xi _t}{|\xi _t|}, \alpha _t dB_t \right\rangle + \frac 1{2|\xi _t|} \left( {\rm tr}( \alpha _t \alpha _t ^T) - \frac{|\alpha _t^T \xi _t|^2}{|\xi _t|^2}\right) dt
\end{equation}
where ${\rm tr}(A)$ is the trace of the matrix $A$.
Now we follow the argument in Section 3 of \cite{LiRo}.
Since
\begin{align*}
\alpha _t &= \sigma (t,X_t^x) - \sigma (t,Z_t^z) + \frac{2 \xi _t\otimes (\sigma (t,Z_t^z)^{-1} \xi _t)}{|\sigma (t,Z_t^z)^{-1} \xi _t|^2}\\
&= \sigma (t,X_t^x) - \sigma (t,Z_t^z) + \frac{2 \xi _t \xi _t^T (\sigma (t,Z_t^z)^{-1} )^T}{|\sigma (t,Z_t^z)^{-1} \xi _t|^2},
\end{align*}
it holds that
\begin{align*}
& {\rm tr}( \alpha _t \alpha _t ^T)  - \frac{|\alpha _t^T \xi _t|^2}{|\xi _t|^2} \\
& = {\rm tr}\left( [\sigma (t,X_t^x) - \sigma (t,Z_t^z)] [\sigma (t,X_t^x) - \sigma (t,Z_t^z)]^T \right) - \frac{\left| [\sigma (t,X_t^x) - \sigma (t,Z_t^z)]^T \xi _t \right| ^2}{|\xi _t|^2}.
\end{align*}
Hence, in view of  (\ref{eq:sigma}), there exists a positive constant $\gamma _1$ depending on $d$ and $\Lambda$ such that
\begin{equation}\label{eq:coup01}
\left| {\rm tr}( \alpha _t \alpha _t ^T) - \frac{|\alpha _t^T \xi _t|^2}{|\xi _t|^2}\right| \leq \gamma _1 \rho _R (|\xi _t|) ,\quad t\in [0,\tau)\ \mbox{such that}\ X_t^x, Z_t^z \in B(0;R),
\end{equation}
On the other hand, following the argument in Section 3 of \cite{LiRo}, we have a positive constant $\gamma _2$ depending on $d$ and $\Lambda$ such that
\begin{equation}\label{eq:coup02}
\frac{|\alpha _t ^T \xi _t|}{|\xi_t|} \geq \gamma_2 ^{-1} \ \mbox{for}\ t\in [0,\tau)\ \mbox{such that}\ |\sigma (t,X_t^x)-\sigma (t,Z_t^z)|\leq 2\Lambda ^{-1}.
\end{equation}
Note that if $\rho _R (|\xi _t|) \leq 2\Lambda ^{-1}$ and $X_t^x, Z_t^z\in B(0;R)$, then $|\sigma (t,X_t^x)-\sigma (t,Z_t^z)|\leq 2\Lambda ^{-1}$.
Let $\gamma := \gamma _1 \vee \gamma _2$.
Define stopping times $\tau _n$ by $\tau _n := \inf \{ t>0; |X_t^x -Z_t^z| \leq 1/n \}$ for $n\in {\mathbb N}$.
For given $\varepsilon >0$, let
\begin{align*}
\tilde \tau &:= \tau \wedge \inf \left\{ t\in [0,\infty);\ \rho _R (|\xi _t|) > \frac{\varepsilon}{2\gamma ^3} \wedge 2\Lambda ^{-1},\ X_t^x \not \in B(0;R) \ \mbox{or}\   Z_t^z \not \in B(0;R) \right\} \\
\tilde \tau _n &:= \tau _n \wedge \inf \left\{ t\in [0,\infty);\ \rho _R (|\xi _t|) > \frac{\varepsilon}{2\gamma ^3} \wedge 2\Lambda ^{-1},\ X_t^x \not \in B(0;R) \ \mbox{or}\   Z_t^z \not \in B(0;R) \right\}
\end{align*}
for $n\in {\mathbb N}$.
Then, it holds that $\tilde \tau _n \uparrow \tilde \tau$ almost surely as $n\rightarrow \infty$.
By It\^o's formula, (\ref{eq:xi}), (\ref{eq:coup01}) and (\ref{eq:coup02}), we have for $t\in [0,\infty )$
\begin{align*}
&E\left[ |\xi _{t\wedge \tilde \tau _n}|^{1-\varepsilon}\right] \\
& = |x-z|^{1-\varepsilon} + (1-\varepsilon ) E\left[ \int _0^{t\wedge \tilde \tau _n} |\xi _s|^{-\varepsilon} \frac 1{2|\xi _s|} \left( {\rm tr}( \alpha _s \alpha _s ^T) - \frac{|\alpha _s^T \xi _s|^2}{|\xi _s|^2}\right) ds\right] \\
& \quad - \frac{\varepsilon (1-\varepsilon )}{2} E\left[ \int _0^{t\wedge \tilde \tau _n} |\xi _s|^{ -1-\varepsilon} \frac{| \alpha _s^T \xi _s|^2}{|\xi _s|^2} ds\right] \\
& \leq |x-z|^{1-\varepsilon} + \frac{(1-\varepsilon )\gamma}2 E\left[ \int _0^{t\wedge \tilde \tau _n} |\xi _s|^{-1-\varepsilon} \rho _R(|\xi _s|)ds \right] - \frac{\varepsilon (1-\varepsilon )}{2\gamma ^2} E\left[ \int _0^{t\wedge \tilde \tau _n} |\xi _s|^{ -1-\varepsilon} ds\right] \\
& \leq |x-z|^{1-\varepsilon} + \frac{\varepsilon (1-\varepsilon )}{4\gamma ^2} E\left[ \int _0^{t\wedge \tilde \tau _n} |\xi _s|^{-1-\varepsilon} ds \right] - \frac{\varepsilon (1-\varepsilon )}{2\gamma ^2} E\left[ \int _0^{t\wedge \tilde \tau _n} |\xi _s|^{ -1-\varepsilon} ds\right] \\
& \leq |x-z|^{1-\varepsilon} - \frac{\varepsilon (1-\varepsilon )}{4\gamma ^2} E\left[ \int _0^{t\wedge \tilde \tau _n} |\xi _s|^{-1-\varepsilon} ds \right] \\
& \leq |x-z|^{1-\varepsilon} - \frac{\varepsilon (1-\varepsilon )}{2^{3+\varepsilon}\gamma ^2R^{1+\varepsilon}} E\left[ t\wedge \tilde \tau _n\right].
\end{align*}
Hence, it holds that
\begin{equation}\label{eq:est-nhit}
E\left[ t\wedge \tilde \tau \right] \leq C|x-z|^{1-\varepsilon} \quad \mbox{for}\ t\in [0,\infty )
\end{equation}
where $C$ is a constant depending on $d$, $\varepsilon$, $R$ and $\Lambda$.

Now we consider the estimate of the expectation of $\tau$ by using that of $\tilde \tau$.
To simplify the notation, let
\[
\delta _0 := \frac 13 \rho _R ^{-1} \left( \frac{\varepsilon}{2\gamma ^3} \wedge 2\Lambda ^{-1}\right) .
\]
Since
\begin{align*}
& |\xi _t| > 3\delta _0 \Longrightarrow |X_t^x -x| >\delta _0 ,\ |Z_t^z -z| >\delta _0,\ \mbox{or}\ |x-z|>\delta _0,\\
& X_t^x \not \in B(0;R) \ \mbox{or}\   Z_t^z \not \in B(0;R) \Longrightarrow |X_t^x -x| >\frac R2 \ \mbox{or}\ |Z_t^z -z| >\frac R2
\end{align*}
we have for $x,z \in B(0;R/2)$ such that $|x-z| \leq \delta _0$,
\begin{equation}\label{eq:estprob1-2}\begin{array}{l}
\displaystyle P(\tau \geq t) \leq P(\tilde \tau \geq t) +P\left( \sup _{s\in [0,t]} |X_s^x -x| >\delta _0 \wedge \frac R2 \right) \\
\displaystyle \hspace{6cm} +P\left( \sup _{s\in [0,t]} |Z_s^z -z| >\delta _0 \wedge \frac R2 \right) .
\end{array}\end{equation}
Let $\eta = x$ or $z$.
By Chebyshev's inequality and Burkholder's inequality we have
\begin{align*}
& P\left( \sup _{s\in [0,t]}|X_s^\eta -\eta | > \delta _0 \wedge \frac R2 \right) \\
& \leq \left(\delta _0 \wedge \frac R2\right)  ^{-2/\varepsilon}  E\left[ \sup _{s\in [0,t]}|X_s^\eta -\eta |^{2/\varepsilon} \right] \\
& \leq \left(\delta _0 \wedge \frac R2\right)  ^{-2/\varepsilon} E\left[ \sup _{s\in [0,t]}\left| \int _0^s \sigma (u,X_u^\eta )dB_u \right|^{2/\varepsilon} \right] \\
& \leq \left(\delta _0 \wedge \frac R2\right)  ^{-2/\varepsilon} C E\left[ \left( \sum _{i,j=1}^d \int _0^t \sigma _{ij}(u,X_u^\eta ) \sigma _{ji}(u,X_u^\eta ) du \right) ^{1/\varepsilon} \right] \\
&\leq d^{1/\varepsilon } \left(\delta _0 \wedge \frac R2\right)  ^{-2/\varepsilon}  C \Lambda ^{1/\varepsilon} t^{1/\varepsilon}
\end{align*}
where $C$ is a constant depending on $\varepsilon $.
Hence, there exists a constant $C$ depending on $d$, $\varepsilon$, $R$, $\rho _R$ and $\Lambda$ such that
\begin{equation}\label{eq:estprob1-1}
P\left( \sup _{s\in [0,t]}|X_s^\eta -\eta | >\delta _0 \wedge \frac R2 \right) \leq C|x-z|
\end{equation}
for $\eta = x,z$ and $t\in [0, |x-z|^\varepsilon ]$.
By (\ref{eq:est-nhit}), (\ref{eq:estprob1-2}) and (\ref{eq:estprob1-1}) we have for $x,z \in B(0;R/2)$ such that $|x-z| \leq \delta _0$, and $t\in [0, |x-z|^\varepsilon ]$
\begin{align*}
&E[t\wedge \tau] \\
&\leq \int _0^t P(\tau \geq s) ds \\
&\leq \int _0^t P( \tilde \tau \geq s) ds + t \left[ P\left( \sup _{s\in [0,t]}|X_s^x -x| >\delta _0 \wedge \frac R2 \right) + P\left( \sup _{s\in [0,t]}|X_s^z -z| >\delta _0 \wedge \frac R2 \right) \right] \\
&\leq C(1+t)|x-z|^{1-\varepsilon}
\end{align*}
where $C$ is a constant depending on $d$, $\varepsilon$, $R$, $\rho _R$ and $\Lambda$.
Therefore, we obtain
\begin{equation}\label{eq:estprob1-9}
E[t\wedge \tau] \leq C(1+t)|x-z|^{1-\varepsilon}
\end{equation}
for $x,z \in B(0;R/2)$ such that $|x-z| \leq \delta _0$, and $t\in [0, |x-z|^\varepsilon ]$.
By using Chebyshev's inequality, calculate $E[t\wedge \tau]$ as 
\begin{align*}
E[t\wedge \tau]
& = \int _0^{|x-z|^\varepsilon} P(\tau \geq s) ds +\int _{|x-z|^\varepsilon}^t P(\tau \geq s) ds\\
& \leq E[|x-z|^\varepsilon \wedge \tau] + t P(\tau \geq |x-z|^\varepsilon )\\
& \leq E[|x-z|^\varepsilon \wedge \tau] + \frac{t}{|x-z|^\varepsilon} E[ \tau \wedge |x-z|^\varepsilon] \\
&\leq \left( 1+ t|x-z|^{-\varepsilon} \right) E[|x-z|^\varepsilon \wedge \tau].
\end{align*}
Thus, applying (\ref{eq:estprob1-9}) with $t=|x-z|^\varepsilon$ and choosing another small $\varepsilon$, we obtain (\ref{eq:estprob1-0}) for all $t \in [0,\infty )$.
\end{proof}

\begin{lem}\label{lem:estprob2}
For $R>0$ and sufficiently small $\varepsilon >0$, there exist positive constants $C$ and $c_0$ depending on $d$, $\varepsilon$, $C_{\rm G}^+$, $R$, $\rho _R$ and $\Lambda$ such that
\begin{align*}
& p^X(0,x;t,y)E^{X_t^x = y}[t\wedge \tau] \leq C t^{-d/2}(1+t^2) |x-z|^{1-\varepsilon} \\
& p^X(0,z;t,y)E^{Z_t^z = y}[t\wedge \tau] \leq C t^{-d/2}(1+t^2) |x-z|^{1-\varepsilon}
\end{align*}
for $t \in (0,\infty )$, $x,z\in B(0;R/2)$ such that $|x-z| \leq c_0$, and $y\in {\mathbb R}^d$.
\end{lem}

\begin{proof}
It holds that
\begin{equation}\label{eq:estprob2-01}
\displaystyle E^{X^x_t=y}\left[ t\wedge \tau \right] = E^{X^x_t=y}\left[ (t\wedge \tau) {\mathbb I}_{[0,t/2]}(\tau) \right] + E^{X^x_t=y}\left[ (t\wedge \tau) {\mathbb I}_{[t/2,\infty)}(\tau) \right] .
\end{equation}
By (\ref{eq:conditional}) and (\ref{eq:gaussest}) we have
\begin{align*}
& p^X(0,x;t,y)E^{X^x_t=y}\left[ (t\wedge \tau) {\mathbb I}_{[0,t/2]}(\tau) \right] \\
& = E\left[ (t\wedge \tau) {\mathbb I}_{[0,t/2]}(\tau)\ p^X\left( \frac t2,X^x _{t/2};t,y\right) \right]\\
& \leq 2^{d/2}C_{\rm G}^+ t^{-d/2}E\left[ t\wedge \tau \right] .
\end{align*}
Hence, in view of Lemma \ref{lem:estprob1}, there exists positive constants $C$ and $c_0$ depending on $d$, $\varepsilon$, $C_{\rm G}^+$, $R$, $\rho _R$ and $\Lambda$ such that
\begin{equation}\label{eq:estprob2-02}
p^X(0,x;t,y)E^{X^x_t=y}\left[ (t\wedge \tau) {\mathbb I}_{[0,t/2]}(\tau) \right] \leq C t^{-d/2} (1+t^2) |x-z|^{1-\varepsilon}
\end{equation}
for $x,z\in B(0;R/2)$ such that $|x-z| \leq c_0$ and $y\in {\mathbb R}^d$.

On the other hand, by (\ref{eq:conditional}) and (\ref{eq:gaussest}) we have
\begin{align*}
& p^X(0,x;t,y)E^{X^x_t=y}\left[ (t\wedge \tau) {\mathbb I}_{[t/2,\infty)}(\tau) \right] \\
& \leq t p^X(0,x;t,y) P^{X^x_t=y}\left( \tau > \frac t2 \right) \\
& = t \int _{{\mathbb R}^d} p^X\left( \frac t2,z;t,y\right) P\left(\tau > \frac t2 ,\  X_{t/2}^x \in dz \right) \\
& \leq 2^{d/2}C_{\rm G}^+ t^{-d/2+1} P\left( \tau > \frac t2 \right) .
\end{align*}
Hence, by applying Chebyshev's inequality we have
\begin{align*}
p^X(0,x;t,y)E^{X^x_t=y}\left[ (t\wedge \tau) {\mathbb I}_{[t/2,\infty)}(\tau) \right] 
& \leq C t^{-d/2} E\left[ t \wedge \tau \right]
\end{align*}
where $C$ is a constant depending on $d$ and $C_{\rm G}^+$.
Thus, Lemma \ref{lem:estprob1} implies that
\begin{equation}\label{eq:estprob2-03}
p^X(0,x;t,y)E^{X^x_t=y}\left[ (t\wedge \tau) {\mathbb I}_{[t/2,\infty)}(\tau) \right] \leq C t^{-d/2}\left( 1+t^2 \right) |x-z|^{1-\varepsilon}
\end{equation}
for $x,z\in B(0;R/2)$ such that $|x-z| \leq c_0$ where $C$ and $c_0$ are positive constants depending on $d$, $\varepsilon$, $C_{\rm G}^+$, $R$, $\rho _R$ and $\Lambda$.
Therefore, we obtain the assertion for $x$ by (\ref{eq:estprob2-01}), (\ref{eq:estprob2-02}) and (\ref{eq:estprob2-03}).
Similar argument yields the the assertion for $z$.
\end{proof}

\begin{lem}\label{lem:II}
For $q\geq 1$, $R>0$ and sufficiently small $\varepsilon >0$, there exist positive constants $C$ and $c_0$ depending on $q$, $d$, $\varepsilon$, $R$, $\rho _R$, $\Lambda$, $\| b\| _\infty$ and $\| c\| _\infty$, such that
\begin{align*}
& E\left[ \sup _{s\in [0,t]} \left| {\mathcal E}(0,\tau \wedge s; X^x) - 1 \right| ^{q} \right]  \leq C e^{Ct}|x-z|^{2/(q\vee2)-\varepsilon} ,\\
& E\left[ \sup _{s\in [0,t]} \left| {\mathcal E}(0,\tau \wedge s; Z^z) - 1 \right| ^{q} \right]  \leq C e^{Ct}|x-z|^{2/(q\vee2)-\varepsilon}
\end{align*}
for $t\in [0,\infty )$, $x,z\in B(0;R/2)$ such that $|x-z| \leq c_0$, and $y \in {\mathbb R}^d$.
\end{lem}

\begin{proof}
By (\ref{eq:ItoE}) we have
\begin{align*}
& E \left[ \sup _{v\in [0,\tau \wedge t]} \left| {\mathcal E}(0,v; X^x) - 1 \right| ^{q} \right] \\
& = E \left[ \sup _{v\in [0,\tau \wedge t]} \left| \int _0^v {\mathcal E}(0,u; X^x) \langle b_\sigma (u,X_u^x), dB_u\rangle + \int _0^v {\mathcal E}(0,u; X^x) c(u,X_u^x)du \right| ^q \right] \\
& \leq C E \left[ \sup _{v\in [0,\tau \wedge t]} \left| \int _0^v {\mathcal E}(0,u; X^x) \langle b_\sigma (u,X_u^x), dB_u\rangle  \right| ^q \right] \\
& \hspace{6cm} + C E \left[ \sup _{v\in [0,\tau \wedge t]} \left| \int _0^v {\mathcal E}(0,u; X^x) c(u,X_u^x)du \right| ^q \right]
\end{align*}
where $C$ is a constant depending on $q$.
The terms of the right-hand side of this inequality are dominated as follows.
By the Burkholder's inequality and H\"older's inequality we have
\begin{align*}
& E \left[ \sup _{v\in [0,\tau \wedge t]} \left| \int _0^v {\mathcal E}(0,u; X^x) \langle b_\sigma (u,X_u^x), dB_u\rangle  \right| ^q \right] \\
& \leq C E \left[ \left( \int _0^{\tau \wedge t} {\mathcal E}(0,u; X^x)^2 \left| b_\sigma (u,X_u^x) \right| ^2 du \right) ^{q/2} \right]\\
& \leq C \Lambda ^q \| b \|_\infty ^q t^{1-{1/[(q/2) \vee 1]}}E \left[ \int _0^{\tau \wedge t} {\mathcal E}(0,u; X^x)^{2[(q/2) \vee 1]} du \right]  ^{1/[(q/2) \vee 1]}\\
& \leq C \Lambda ^q \| b \|_\infty ^q t^{1-2/[q \vee 2]} E \left[ \tau \wedge t \right] ^{2(1-\varepsilon )/(q\vee 2)} E \left[ \left( \int _0^t {\mathcal E}(0,u; X^x)^{(q \vee 2)/\varepsilon} du \right) \right]^{2\varepsilon /(q\vee 2)}
\end{align*}
where $C$ is a constant depending on $q$, and by the H\"older's inequality we have
\begin{align*}
& E \left[ \sup _{v\in [0,\tau \wedge t]} \left| \int _0^v {\mathcal E}(0,u; X^x) c(u,X_u^x)du \right|^q \right] \\
& \leq \| c\| _\infty ^q t^{1-1/q} E \left[ \int _0^{\tau \wedge t} {\mathcal E}(0,u; X^x) ^q du \right] \\
& \leq \| c\| _\infty ^q t^{1-1/q} E \left[ \tau \wedge t \right] ^{1-\varepsilon} E \left[ \int _0^t {\mathcal E}(0,u; X^x)^{q/\varepsilon}du \right] ^{\varepsilon } .
\end{align*}
Thus, applying by Lemmas \ref{lem:estE0} and \ref{lem:estprob1} to these inequalities and choosing another small $\varepsilon$, we obtain
\[
E \left[ \sup _{v\in [0,\tau \wedge t]} \left| {\mathcal E}(0,v; X^x) - 1 \right| ^q \right] \leq Ce^{Ct}|x-z|^{2/(q\vee2)-\varepsilon}
\]
where $C$ is a constant depending on $q$, $d$, $\varepsilon$, $R$, $\rho _R$, $\Lambda$, $\| b\| _\infty$ and $\| c\| _\infty$.
Similar argument yields the same estimate for $Z^z$.
\end{proof}

Now we start the proof of Proposition \ref{prop1}.
Let $t\in (0,\infty)$, $x,z\in B(0;R/2)$ such that $x\neq z$, $y \in {\mathbb R}^d$ and $s\in (t/2,t)$.
Recall that $X^z$ and $Z^z$ have the same law.
By (\ref{eq:conditional}) and (\ref{eq:lemII01})  we have
\begin{align*}
& \left| p^X(0,x;t,y)E^{X_t^x=y}\left[ {\mathcal E}(0,s; X^x) ;\ \tau \leq \frac t2 \right] - p^X(0,z;t,y) E^{Z_t^z=y}\left[ {\mathcal E}(0,s; Z^z) ;\ \tau \leq \frac t2\right] \right| \\
& =\left| E\left[ {\mathcal E}(0,s; X^x) p^X(s,X^x_s ;t,y) ;\ \tau \leq \frac t2 \right] - E\left[ {\mathcal E}(0,s; Z^z) p^X(s,Z^z_s ;t,y)  ;\ \tau \leq \frac t2 \right] \right| \\
& \leq E\left[ {\mathcal E}(0,s; Z^z) \left| p^X(s,X^x_s ;t,y) - p^X(s,Z^z_s ;t,y) \right|  ;\ \tau \leq \frac t2 \right] \\
& \quad + E\left[ \left| {\mathcal E}(0,\tau \wedge s; X^x) - {\mathcal E}(0,\tau \wedge s; Z^z) \right| {\mathcal E}(\tau \wedge s ,s; Z^z) p^X(s,X^x_s ;t,y) ;\ \tau \leq \frac t2 \right] \\
& \quad +E\left[ {\mathcal E}(0,\tau \wedge s; X^x) \left|{\mathcal E}(\tau \wedge s ,s; X^x) - {\mathcal E}(\tau \wedge s ,s; Z^z) \right| p^X(s,X^x_s ;t,y)  ;\ \tau \leq \frac t2 \right]
\end{align*}
Noting that
\[ 
X^x_s=Z^z_s\quad \mbox{for}\ s\geq \tau,
\]
we obtain
\begin{equation}\label{eq:inital01}\begin{array}{l}
\displaystyle\left| p^X(0,x;t,y)E^{X_t^x=y}\left[ {\mathcal E}(0,s; X^x) ;\ \tau \leq \frac t2 \right] \right. \\
\displaystyle \hspace{5cm} \left.- p^X(0,z;t,y) E^{Z_t^z=y}\left[ {\mathcal E}(0,s; Z^z) ;\ \tau \leq \frac t2\right] \right| \\
\displaystyle \leq E\left[ \left| {\mathcal E}(0,\tau \wedge s; X^x) - {\mathcal E}(0,\tau \wedge s; Z^z) \right| \phantom{\frac t2} \right. \\
\displaystyle \quad \hspace{5cm} \left. \times {\mathcal E}(\tau \wedge s ,s; Z^z) p^X(s,X^x_s ;t,y)  ;\ \tau \leq \frac t2 \right] .
\end{array}\end{equation}
By the triangle inequality and H\"older's inequality we obtain
\begin{align*}
& E\left[ \left| {\mathcal E}(0,\tau \wedge s; X^x) - {\mathcal E}(0,\tau \wedge s; Z^z) \right| {\mathcal E}(\tau \wedge s ,s; Z^z) p^X(s,X^x_s ;t,y)  ;\ \tau \leq \frac t2 \right] \\
& \leq E\left[ \left| {\mathcal E}(0,\tau \wedge s; X^x) - 1 \right| {\mathcal E}(\tau \wedge s ,s; Z^z) p^X(s,X^x_s ;t,y)  ;\ \tau \leq \frac t2 \right] \\
&\quad + E\left[ \left| {\mathcal E}(0,\tau \wedge s; Z^z) - 1 \right| {\mathcal E}(\tau \wedge s ,s; Z^z) p^X(s,X^x_s ;t,y)  ;\ \tau \leq \frac t2 \right] \\
& \leq \left( E\left[ \left| {\mathcal E}(0,\tau \wedge s; X^x) - 1 \right| ^{2/(2-\varepsilon )} p^X(s,X^x_s ;t,y)  ;\ \tau \leq \frac t2 \right] ^{1-\varepsilon /2} \right.\\
&\hspace{2cm} \left. + E\left[ \left| {\mathcal E}(0,\tau \wedge s; Z^z) - 1 \right| ^{2/(2-\varepsilon )} p^X(s,X^x_s ;t,y)  ;\ \tau \leq \frac t2 \right] ^{1-\varepsilon /2} \right) \\
&\hspace{1cm} \times E\left[ {\mathcal E}(\tau \wedge s ,s; Z^z) ^{2/\varepsilon} p^X(s,X^x_s ;t,y);\ \tau \leq \frac t2 \right] ^{\varepsilon /2}.
\end{align*}
Hence, by (\ref{eq:conditional}) and (\ref{eq:gaussest}) we have
\begin{align*}
& E\left[ \left| {\mathcal E}(0,\tau \wedge s; X^x) - {\mathcal E}(0,\tau \wedge s; Z^z) \right| {\mathcal E}(\tau \wedge s ,s; Z^z) p^X(s,X^x_s ;t,y)  ;\ \tau \leq \frac t2 \right] \\
& \leq \left( E^{X_t^x=y} \left[ \left| {\mathcal E}(0,\tau \wedge s; X^x) - 1 \right| ^{2/(2-\varepsilon )} ;\ \tau \leq \frac t2 \right] ^{1-\varepsilon /2} \right. \\
& \hspace{2cm} \left. + E^{X_t^x=y} \left[ \left| {\mathcal E}(0,\tau \wedge s; Z^z) - 1 \right| ^{2/(2-\varepsilon )}  ;\ \tau \leq \frac t2 \right] ^{1-\varepsilon /2} \right) \\
& \hspace{1cm} \times p^X(0,x;t,y) E^{X_t^x=y} \left[ {\mathcal E}(\tau \wedge s ,s; Z^z) ^{2/\varepsilon} ;\ \tau \leq \frac t2 \right] ^{\varepsilon /2}\\
& \leq \left( E \left[ \left| {\mathcal E}(0,\tau \wedge s; X^x) - 1 \right| ^{2/(2-\varepsilon )} p^X(t/2,X^x_{t/2} ;t,y) ;\ \tau \leq \frac t2 \right] ^{1-\varepsilon /2} \right. \\
& \hspace{2cm} \left. + E \left[ \left| {\mathcal E}(0,\tau \wedge s; Z^z) - 1 \right| ^{2/(2-\varepsilon )} p^X(t/2,X^x_{t/2} ;t,y) ;\ \tau \leq \frac t2 \right] ^{1-\varepsilon /2} \right) \\
& \hspace{1cm} \times \left( p^X(0,x;t,y) E^{X_t^x=y} \left[ {\mathcal E}(\tau \wedge s ,s; Z^z) ^{2/\varepsilon} ;\ \tau \leq \frac t2 \right] \right) ^{\varepsilon /2}\\
& \leq (C_{\rm G}^+ ) ^{1-\varepsilon /2} t^{-d/2 + d\varepsilon /4} \left( p^X(0,x;t,y) E^{X_t^x=y} \left[ {\mathcal E}(\tau \wedge s ,s; Z^z) ^{2/\varepsilon} \right] \right) ^{\varepsilon /2} \\
& \quad \times \left( E \left[ \left| {\mathcal E}(0,\tau \wedge s; X^x) - 1 \right| ^{2/(2-\varepsilon )} \right] ^{1-\varepsilon /2}  + E \left[ \left| {\mathcal E}(0,\tau \wedge s; Z^z) - 1 \right| ^{2/(2-\varepsilon )} \right] ^{1-\varepsilon /2} \right) .
\end{align*}
Applying Lemmas \ref{lem:estE0} and \ref{lem:II} to this inequality, we obtain
\begin{equation}\label{eq:inital03} \begin{array}{l}
\displaystyle E\left[ \left| {\mathcal E}(0,\tau \wedge s; X^x) - {\mathcal E}(0,\tau \wedge s; Z^z) \right| {\mathcal E}(\tau \wedge s ,s; Z^z) p^X(s,X^x_s ;t,y)  ;\ \tau \leq \frac t2 \right] \\
\displaystyle \leq C t^{-d/2} e^{Ct} |x-z|^{1-\varepsilon}
\end{array}\end{equation} 
for $x,z\in B(0;R/2)$ such that $|x-z| \leq c_0$, where $C$ and $c_0$ are constants depending on $d$, $\varepsilon$, $C_{\rm G}^+$,  $R$, $\rho _R$, $\Lambda$, $\| b\| _\infty$ and $\| c\| _\infty$.

H\"older's inequality and Chebyshev's inequality imply
\begin{align*}
E^{X_t^x=y}\left[ {\mathcal E}(0,s; X^x) ;\ \tau \geq \frac t2 \right] &\leq P^{X_t^x=y} \left( \tau \geq \frac t2\right) ^{1-\varepsilon /2} E^{X_t^x=y}\left[ {\mathcal E}(0,s; X^x)^{2/\varepsilon} \right] ^{\varepsilon /2}\\
&\leq \frac{2^{1-\varepsilon /2}}{t^{1-\varepsilon /2}} E^{X_t^x=y} \left[ \tau \wedge t \right] ^{1-\varepsilon /2} E^{X_t^x=y}\left[ {\mathcal E}(0,s; X^x)^{2/\varepsilon} \right] ^{\varepsilon /2}.
\end{align*}
Hence, by Lemmas \ref{lem:estE2} and \ref{lem:estprob2} we obtain
\begin{equation}\label{eq:inital04}
p^X(0,x;t,y) E^{X_t^x=y}\left[ {\mathcal E}(0,s; X^x) ;\ \tau \geq \frac t2 \right] \leq C t^{-d/2-1} e^{Ct} |x-z|^{1-\varepsilon}
\end{equation}
for $x,z\in B(0;R/2)$ such that $|x-z| \leq c_0$, where $C$ and $c_0$ are constants depending on $d$, $\varepsilon$, $m$, $M$, $\theta$, $R$, $\rho _R$, $\Lambda$, $\| b\| _\infty$ and $\| c\| _\infty$.
Similarly we have
\begin{equation}\label{eq:inital04-2}
p^X(0,z;t,y) E^{Z_t^z=y}\left[ {\mathcal E}(0,s; Z^z) ;\ \tau \geq \frac t2 \right] \leq C t^{-d/2-1} e^{Ct} |x-z|^{1-\varepsilon}
\end{equation}
for $x,z\in B(0;R/2)$ such that $|x-z| \leq c_0$, where $C$ and $c_0$ are constants depending on $d$, $\varepsilon$,$\gamma_{\rm G}^-$, $\gamma_{\rm G}^+$, $C_{\rm G}^-$, $C_{\rm G}^+$, $m$, $M$, $\theta$, $R$, $\rho _R$, $\Lambda$, $\| b\| _\infty$ and $\| c\| _\infty$.
Thus, (\ref{eq:fundamental}), (\ref{eq:inital01}), (\ref{eq:inital03}), (\ref{eq:inital04}) and (\ref{eq:inital04-2}) imply
\[
|p(0,x;t,y)-p(0,z;t,y)| \leq C t^{-d/2-1+\varepsilon /2} e^{Ct} |x-z|^{1-\varepsilon}
\]
for $t\in (0,\infty)$, $x,z\in B(0;R/2)$ such that $|x-z| \leq c_0$, and $y \in {\mathbb R}^d$ with constants $C$ and $c_0$ depending on $d$, $\varepsilon$, $\gamma_{\rm G}^-$, $\gamma_{\rm G}^+$, $C_{\rm G}^-$, $C_{\rm G}^+$, $m$, $M$, $\theta$, $R$, $\rho _R$, $\Lambda$, $\| b\| _\infty$ and $\| c\| _\infty$.
By (\ref{eq:gaussest}) we can remove the restriction on $|x-z|$, and therefore, we obtain Proposition \ref{prop1}.

\section{The case of general $a$ (Proof of the main theorem)}\label{sec:proof}

Let $a^{(n)}(t,x)=(a_{ij}^{(n)}(t,x))$ be the symmetric $d\times d$-matrix-valued bounded measurable functions on $[0,\infty ) \times {\mathbb R}^d$ which converge to $a(t,x)$ for each $(t,x)\in [0,\infty ) \times {\mathbb R}^d$ and satisfy (\ref{ass:a1}), (\ref{ass:a3}) and (\ref{ass:a2}).
Consider the following parabolic partial differential equation
\begin{equation}\label{PDEn}
\left\{ \begin{array}{rl}
\displaystyle \frac{\partial}{\partial t} u(t,x) & \displaystyle = \frac 12 \sum _{i,j=1}^d a_{ij}^{(n)}(t,x)\frac{\partial ^2}{\partial x_i \partial x_j} u(t,x) +  \sum _{i=1}^d b_i(t,x) \frac{\partial }{\partial x_i} u(t,x) + c(t,x) u(t,x) \!\!\!\!\!\!\!\!\!\! \\[3mm]
\displaystyle u(0,x)& =f(x)\displaystyle 
\end{array}\right.
\end{equation}
Denote the fundamental solution to (\ref{PDEn}) by $p^{(n)}(s,x;t,y)$.
From (\ref{eq:gaussest}) and Proposition \ref{lem:cpt} we have positive constants $\gamma_1$, $\gamma_2$, $C_1$ and $C_2$ depending on $d$, $\gamma_{\rm G}^-$, $\gamma_{\rm G}^+$, $C_{\rm G}^-$, $C_{\rm G}^+$, $m$, $M$, $\theta$, $\Lambda$, $\| b\| _\infty$ and $\| c\| _\infty$, such that
\begin{equation}\label{eq:gaussestn}
\frac{C_1e^{-C_1 (t-s)}}{(t-s)^{\frac d2}} \exp \left( -\frac{\gamma_1 |x-y|^2}{t-s}\right) \leq p^{(n)}(s,x;t,y) \leq \frac{C_2e^{C_2 (t-s)}}{(t-s)^{\frac d2}} \exp \left( -\frac{\gamma_2 |x-y|^2}{t-s}\right)
\end{equation}
for $s,t \in [0,\infty)$ such that $s<t$, $x,y\in {\mathbb R}^d$ and $n\in {\mathbb N}$.

It is known that the local H\"older continuity of the fundamental solution follows with the index and the constant depending only on the constants appeared in the Gaussian estimate (see \cite{St}).
This fact and (\ref{eq:gaussestn}) imply that the Arzel\'a-Ascoli theorem is applicable to $p^{(n)}$.
Moreover, in view of Proposition \ref{prop1}, there exists a constant $C$ depending on $d$, $\varepsilon$, $\gamma_{\rm G}^-$, $\gamma_{\rm G}^+$, $C_{\rm G}^-$, $C_{\rm G}^+$, $m$, $M$, $\theta$, $R$, $\rho _R$, $\Lambda$, $\| b\| _\infty$ and $\| c\| _\infty$ such that
\[
|p^{(n)}(0,x; t,y) - p^{(n)}(0,z; t,y)| \leq Ct^{-d/2-1} e^{Ct} |x-z|^{1 -\varepsilon} 
\]
for $t\in (0,\infty )$, $y\in {\mathbb R}^d$ and $x,z \in B(0;R/2)$.
Hence, there exists a continuous function $p^{(\infty)}(0,\cdot ;\cdot ,\cdot )$ on ${\mathbb R}^d \times (0,\infty) \times {\mathbb R}^d$ such that
\begin{align}
&\lim _{n\rightarrow \infty} \sup _{|x|\leq R/2} | p^{(n)}(0, x; t,y) -p^{(\infty)}(0,x; t,y)| =0 \\
& |p^{(\infty)}(0,x;t,y) - p^{(\infty)}(0,z;t,y)| \leq Ct^{-d/2-1} e^{Ct} |x-z|^{1 -\varepsilon}, \quad x,z\in B(0;R/2),
\end{align}
for $t\in (0,\infty)$ and $y\in {\mathbb R}^d$ where $C$ is a constant depending on $d$, $\varepsilon$, $\gamma_{\rm G}^-$, $\gamma_{\rm G}^+$, $C_{\rm G}^-$, $C_{\rm G}^+$, $m$, $M$, $\theta$, $R$, $\rho _R$, $\Lambda$, $\| b\| _\infty$ and $\| c\| _\infty$.
Moreover we have positive constants $C_1$, $C_2$, $\gamma _1$ and $\gamma _2$ depending on $d$, $\gamma_{\rm G}^-$, $\gamma_{\rm G}^+$, $C_{\rm G}^-$, $C_{\rm G}^+$, $m$, $M$, $\theta$, $\Lambda$, $\| b\| _\infty$ and $\| c\| _\infty$ such that
\[
\frac{C_1e^{-C_1 (t-s)}}{(t-s)^{\frac d2}} \exp \left( -\frac{\gamma_1 |x-y|^2}{t-s}\right) \leq p^{(\infty)}(s,x;t,y) \leq \frac{C_2e^{C_2 (t-s)}}{(t-s)^{\frac d2}} \exp \left( -\frac{\gamma_2 |x-y|^2}{t-s}\right)
\]
for $s,t \in [0,\infty)$ such that $s<t$, and $x,y\in {\mathbb R}^d$.
To prove Theorem \ref{thm:main} we show that $p^{(\infty)}(0,\cdot ; \cdot ,\cdot)$ coincides with the fundamental solution $p(0,\cdot ; \cdot ,\cdot)$ of the original parabolic partial differential equation (\ref{PDE}).
Let $\phi ,\psi \in C_0^\infty({\mathbb R}^d)$, and denote
\begin{align*}
P_t^{(n)} g(x) &:= \int _{{\mathbb R}^d} g(y)p^{(n)}(0,x;t,y) dy \quad \mbox{for}\ g\in C_b({\mathbb R}^d)\\
L_t^{(n)} &:= \frac 12 \sum _{i,j=1}^d a_{ij}^{(n)}(t,x)\frac{\partial ^2}{\partial x_i \partial x_j} + \sum _{i=1}^d b_i(t,x) \frac{\partial }{\partial x_i} + c(t,x) .
\end{align*}
Noting that $p^{(n)}(s,x;t,y)$ is smooth in $(s,x,t,y)$, we have $P_t^{(n)} L_t^{(n)} \phi (x) = \frac{\partial}{\partial t} P_t^{(n)} \phi (x)$.
Hence, 
\begin{align*}
& \int _{{\mathbb R}^d} \left( \int _{{\mathbb R}^d} \phi (y) p^{(n)}(0,x;t,y) dy\right) \psi (x) dx - \int _{{\mathbb R}^d} \phi (x) \psi (x) dx\\
& = \int _{{\mathbb R}^d} [ P_t^{(n)} \phi (x)] \psi (x) dx - \int _{{\mathbb R}^d} \phi (x) \psi (x) dx \\
& = \int _0^t \int _{{\mathbb R}^d} \left[ \frac{\partial}{\partial s}P_s^{(n)} \phi (x)\right] \psi (x) dx ds\\
& = \int _0^t \int _{{\mathbb R}^d} \left[ P_s^{(n)} L_s^{(n)} \phi (x)\right] \psi (x) dx ds\\
& = \int _0^t \int _{{\mathbb R}^d} \left( \int _{{\mathbb R}^d} \left[ \frac 12 \sum _{i,j=1}^d a_{ij}^{(n)}(s,y)\frac{\partial ^2}{\partial y_i \partial y_j} + \sum _{i=1}^d b_i(s,y) \frac{\partial }{\partial y_i} + c(s,y)\right] \phi (y) \right. \\
& \hspace{6cm}\left. \phantom{\left[ \frac 12 \sum _{i,j=1}^d \right]} \times p^{(n)}(0,x;s,y) dy \right) \psi (x) dx ds.
\end{align*}
Taking the limit as $n \rightarrow \infty$, we obtain
\begin{align*}
& \int _{{\mathbb R}^d} \left( \int _{{\mathbb R}^d} \phi (y) p^{(\infty)}(0,x;t,y) dy\right) \psi (x) dx - \int _{{\mathbb R}^d} \phi (x) \psi (x) dx \\
& = \int _0^t \int _{{\mathbb R}^d} \left( \int _{{\mathbb R}^d} \left[ \frac 12 \sum _{i,j=1}^d a_{ij}(s,y)\frac{\partial ^2}{\partial y_i \partial y_j} + \sum _{i=1}^d b_i(s,y) \frac{\partial }{\partial y_i} + c(s,y) \right] \phi (y) \right. \\
& \hspace{6cm}\left. \phantom{\left[ \frac 12 \sum _{i,j=1}^d \right]} \times p^{(\infty)}(0,x;s,y) dy \right) \psi (x) dx ds.
\end{align*}
This equality implies that $p^{(\infty)}(0,x;t,y)$ is also the fundamental solution to the parabolic partial differential equation (\ref{PDE}).
Since the weak solution to (\ref{PDE}) has the uniqueness, $p^{(\infty)}(0,x;t,y)$ coincides with $p(0,x;t,y)$.
Therefore, we obtain Theorem \ref{thm:main}.

\section*{Acknowledgment}
The author is grateful to Professor Felix Otto for finding a mistake and giving the author the information on Schauder's estimate.
The author also thank to the anonymous referee for careful reading, indicating mistakes of the previous version and giving the author the information of references.
This work was supported by JSPS KAKENHI Grant number 25800054.

\bibliographystyle{plain}
\bibliography{Sei010.bib}

\def\cprime{$'$} \def\cprime{$'$}
\begin{thebibliography}{10}

\bibitem{Ar}
D.~G. Aronson.
\newblock Bounds for the fundamental solution of a parabolic equation.
\newblock {\em Bull. Amer. Math. Soc.}, 73:890--896, 1967.

\bibitem{BoKrRo}
V.~I. Bogachev, N.~V. Krylov, and M.~R{\"o}ckner.
\newblock Elliptic and parabolic equations for measures.
\newblock {\em Russian Math. Surveys}, 64(6):973--1078, 2009.

\bibitem{BoRoSh}
V.~I. Bogach{\"e}v, M.~R{\"o}ckner, and S.~V. Shaposhnikov.
\newblock Global regularity and estimates for solutions of parabolic equations.
\newblock {\em Teor. Veroyatn. Primen.}, 50(4):652--674, 2005.

\bibitem{ChKu}
Zhen-Qing Chen and Takashi Kumagai.
\newblock Heat kernel estimates for stable-like processes on {$d$}-sets.
\newblock {\em Stochastic Process. Appl.}, 108(1):27--62, 2003.

\bibitem{Cr}
M.~Cranston.
\newblock Gradient estimates on manifolds using coupling.
\newblock {\em J. Funct. Anal.}, 99(1):110--124, 1991.

\bibitem{DG}
Ennio De~Giorgi.
\newblock Sulla differenziabilit\`a e l'analiticit\`a delle estremali degli
  integrali multipli regolari.
\newblock {\em Mem. Accad. Sci. Torino. Cl. Sci. Fis. Mat. Nat. (3)}, 3:25--43,
  1957.

\bibitem{Es}
Luis Escauriaza.
\newblock Bounds for the fundamental solution of elliptic and parabolic
  equations in nondivergence form.
\newblock {\em Comm. Partial Differential Equations}, 25(5-6):821--845, 2000.

\bibitem{FaKe}
Eugene~B. Fabes and Carlos~E. Kenig.
\newblock Examples of singular parabolic measures and singular transition
  probability densities.
\newblock {\em Duke Math. J.}, 48(4):845--856, 1981.

\bibitem{Fr}
Avner Friedman.
\newblock {\em Partial differential equations of parabolic type}.
\newblock Prentice-Hall Inc., Englewood Cliffs, N.J., 1964.

\bibitem{IW}
Nobuyuki Ikeda and Shinzo Watanabe.
\newblock {\em Stochastic differential equations and diffusion processes},
  volume~24 of {\em North-Holland Mathematical Library}.
\newblock North-Holland Publishing Co., Amsterdam, second edition, 1989.

\bibitem{KaSh}
Ioannis Karatzas and Steven~E. Shreve.
\newblock {\em Brownian motion and stochastic calculus}, volume 113 of {\em
  Graduate Texts in Mathematics}.
\newblock Springer-Verlag, New York, second edition, 1991.

\bibitem{Kar}
Stefan Karrmann.
\newblock Gaussian estimates for second-order operators with unbounded
  coefficients.
\newblock {\em J. Math. Anal. Appl.}, 258(1):320--348, 2001.

\bibitem{Kr}
N.~V. Krylov.
\newblock {\em Lectures on elliptic and parabolic equations in {H}\"older
  spaces}, volume~12 of {\em Graduate Studies in Mathematics}.
\newblock American Mathematical Society, Providence, RI, 1996.

\bibitem{KuSt}
S.~Kusuoka and D.~Stroock.
\newblock Applications of the {M}alliavin calculus. {II}.
\newblock {\em J. Fac. Sci. Univ. Tokyo Sect. IA Math.}, 32(1):1--76, 1985.

\bibitem{LSU}
O.~A. Lady{\v{z}}enskaja, V.~A. Solonnikov, and N.~N. Ural{\cprime}ceva.
\newblock {\em Linear and quasilinear equations of parabolic type}.
\newblock Translated from the Russian by S. Smith. Translations of Mathematical
  Monographs, Vol. 23. American Mathematical Society, Providence, R.I., 1968.

\bibitem{LM}
P.~D. Lax and A.~N. Milgram.
\newblock Parabolic equations.
\newblock In {\em Contributions to the theory of partial differential
  equations}, Annals of Mathematics Studies, no. 33, pages 167--190. Princeton
  University Press, Princeton, N. J., 1954.

\bibitem{LiRo}
Torgny Lindvall and L.~C.~G. Rogers.
\newblock Coupling of multidimensional diffusions by reflection.
\newblock {\em Ann. Probab.}, 14(3):860--872, 1986.

\bibitem{MePaRh}
G.~Metafune, D.~Pallara, and A.~Rhandi.
\newblock Global properties of transition probabilities of singular diffusions.
\newblock {\em Teor. Veroyatn. Primen.}, 54(1):116--148, 2009.

\bibitem{Na}
J.~Nash.
\newblock Continuity of solutions of parabolic and elliptic equations.
\newblock {\em Amer. J. Math.}, 80:931--954, 1958.

\bibitem{PrEi}
F.~O. Porper and S.~D. {\`E}{\u\i}del{\cprime}man.
\newblock Two-sided estimates of the fundamental solutions of second-order
  parabolic equations and some applications of them.
\newblock {\em Russian Math. Surveys}, 39(3):119--178, 1984.

\bibitem{PrEi2}
F.~O. Porper and S.~D. {\`E}{\u\i}del{\cprime}man.
\newblock Properties of solutions of second-order parabolic equations with
  lower-order terms.
\newblock {\em Trans. Moscow Math. Soc}, pages 101--137, 1993.

\bibitem{Pr}
Philip~E. Protter.
\newblock {\em Stochastic integration and differential equations}, volume~21 of
  {\em Stochastic Modelling and Applied Probability}.
\newblock Springer-Verlag, Berlin, 2005.
\newblock Second edition. Version 2.1, Corrected third printing.

\bibitem{ReYo}
Daniel Revuz and Marc Yor.
\newblock {\em Continuous martingales and {B}rownian motion}, volume 293 of
  {\em Grundlehren der Mathematischen Wissenschaften [Fundamental Principles of
  Mathematical Sciences]}.
\newblock Springer-Verlag, Berlin, third edition, 1999.

\bibitem{St}
Daniel~W. Stroock.
\newblock Diffusion semigroups corresponding to uniformly elliptic divergence
  form operators.
\newblock In {\em S\'eminaire de {P}robabilit\'es, {XXII}}, volume 1321 of {\em
  Lecture Notes in Math.}, pages 316--347. Springer, Berlin, 1988.

\bibitem{StVa}
Daniel~W. Stroock and S.~R.~Srinivasa Varadhan.
\newblock {\em Multidimensional diffusion processes}, volume 233 of {\em
  Grundlehren der Mathematischen Wissenschaften [Fundamental Principles of
  Mathematical Sciences]}.
\newblock Springer-Verlag, Berlin, 1979.

\end{thebibliography}

\end{document}